\documentclass[11pt, a4paper]{article}

\usepackage[top=2.6cm, bottom=2.6cm, left=2.5cm, right=2.5cm]{geometry}

\usepackage{latexsym,amssymb}
\usepackage{amsmath,amsthm}
\usepackage{color}
\usepackage{mathtools}
\usepackage[abbrev]{amsrefs}

\usepackage{multirow,eepic}

\newtheorem{theorem}{\bf Theorem}[section]
\newtheorem{lemma}{\bf Lemma}[section]
\newtheorem{proposition}{\bf Proposition}[section]

\newtheorem{remark}{\bf Remark}[section]

\allowdisplaybreaks[4]

\numberwithin{equation}{section}
\begin{document}
\title{
  Uniform boundedness and blow-up rate of solutions \\ 
  in non-scale-invariant superlinear heat equations
}
\author{
        Yohei Fujishima\footnote{e-mail address: fujishima@shizuoka.ac.jp} \\ \\
        {\small Department of Mathematical and Systems Engineering} \\
        {\small Faculty of Engineering, Shizuoka University} \\
        {\small 3--5--1 Johoku, Hamamatsu 432--8561, Japan} \\ \\
        Toru Kan\footnote{e-mail address: kan@omu.ac.jp} \\ \\
        {\small Department of Mathematics, Osaka Metropolitan University} \\
        {\small 1--1 Gakuen-cho, Sakai 599--8531, Japan}
        }
\date{}
\pagestyle{myheadings}
\markboth{Y. Fujishima and T.~Kan}{Bounds for a superlinear heat equation}

\maketitle

\begin{abstract}
  For superlinear heat equations with the Dirichlet boundary condition, 
  the $L^\infty$ estimates of radially symmetric solutions are studied. 
  In particular,  
  the uniform boundedness of global solutions 
  and the non-existence of solutions with type II blow-up are proved.
  For the space dimension greater than $9$,
  our results are shown under the condition that
  an exponent representing the growth rate of a nonlinear term
  is between the Sobolev exponent and the Joseph-Lundgren exponent.
  In the case where the space dimension is greater than $2$ and smaller than $10$,
  our results are applicable for nonlinear terms 
  growing extremely faster than the exponential function.
\end{abstract}
\vspace{5pt}

\ \\

\section{Introduction}
We are concerned with the behavior of solutions of the semilinear heat equation
\begin{equation}
  \label{eq:1.1}
  \left\{
    \begin{array}{ll}
      \partial_t u = \Delta u + f(u), & x\in\Omega, \,\,\, t>0, 
      \\[3pt]
      u(x,t) = k, & x\in \partial\Omega, \,\,\, t>0, 
      \\[3pt]
      u(x,0) = u_0(x), & x\in \Omega, 
    \end{array}
  \right.
\end{equation}
where $u_0 \in L^\infty(\Omega)$ is a non-negative initial function,
$k$ is a non-negative constant
and $\Omega$ is a smooth bounded domain in $\mathbb{R}^N$, $N\ge 1$.
Throughout the paper,
the nonlinear term $f$ is assumed to be of class $C^{1,\alpha}$ 
on $[0,\infty)$ for some $\alpha \in (0,1)$,
positive in $(0,\infty)$ and superlinear in the sense that
\begin{equation}
  \label{eq:1.2}
  F(u):=\int_u^\infty \frac{1}{f(\eta)}\, d\eta<\infty
  \quad \mbox{for} \quad 
  u>0.
\end{equation}
Then it is well-known that \eqref{eq:1.1} admits a unique local-in-time solution $u(x,t)$
which is positive for $x \in \Omega$, $t>0$.

One of interesting features of superlinear parabolic equations of the form \eqref{eq:1.1}
is that a solution $u$ can become unbounded in finite or infinite time.
Investigation of the behavior of the $L^\infty$-norm $\|u(\cdot,t)\|_{L^\infty(\Omega)}$
is therefore fundamental to understand
whether the solution stays bounded and how it becomes unbounded.
In this paper, 
we discuss two problems on the behavior of $\|u(\cdot,t)\|_{L^\infty(\Omega)}$.
The first problem is the uniform boundedness of global solutions:
we consider whether it is true that
\begin{equation}
  \label{eq:1.3}
  \limsup\limits_{t\to \infty} \|u(\cdot,t)\|_{L^\infty(\Omega)}<\infty
\end{equation} 
as long as $u$ exists globally.
As the second problem,
we also consider blow-up rates of solutions.
More precisely, 
we discuss whether it is true that
\begin{equation}
  \label{eq:1.4}
  \liminf\limits_{t\to T} \frac{F(\|u(\cdot,t)\|_{L^\infty(\Omega)})}{T-t}>0
\end{equation}
when $u$ blows up at some finite time $T$.
This inequality means that the blow-up rate of $u$
is similar to that of a solution of the ordinary differential equation 
\begin{equation*}
  \frac{d\varphi}{dt}=f(\varphi),
  \qquad
  \lim\limits_{t\to T} \varphi(t)=\infty,
\end{equation*}
because the solution $\varphi$ is implicitly given by $F(\varphi(t))/(T-t)=1$.
We note that, 
for the specific nonlinear term $f(u)=u^p$ (resp. $f(u)=e^u$),
\eqref{eq:1.4} is written as
\begin{gather*}
  \limsup\limits_{t\to T} \, (T-t)^{\frac{1}{p-1}} \|u(\cdot,t)\|_{L^\infty(\Omega)} <\infty
  \,\,\, 
  \left( 
    \mbox{resp.} \,\,\, 
    \limsup\limits_{t\to T} 
    \left(
      \log (T-t) +\|u(\cdot,t)\|_{L^\infty(\Omega)} 
    \right)
    <\infty
  \right).
\end{gather*}
Blow-up with this property is referred to as type I blow-up;
if the above limit supremum is infinity,
then blow-up is said to be of type II.
Therefore what we are concerned with in the second problem is 
the non-existence of solutions with type II blow-up.

\subsection{Known results and motivation}

When the nonlinear term $f(u)$ behaves similarly to $u^p$ at infinity,
the so-called Sobolev exponent $p_{\text{S}}$
and Joseph--Lundgren exponent $p_{\text{JL}}$ 
are known to determine the range of $p$ 
such that \eqref{eq:1.3} or \eqref{eq:1.4} is true.
They are given by
\begin{equation*}
  p_{\text{S}}:=
  \left\{
  \begin{aligned}
    &\infty
    &&\mbox{if} \,\,\, N \le 2,
    \\
    &\frac{N+2}{N-2}
    &&\mbox{if} \,\,\, N \ge 3,
  \end{aligned}
  \right.
  \qquad
  p_{\text{JL}}:=
  \left\{
  \begin{aligned}
    &\infty
    &&\mbox{if} \,\,\, N \le 10,
    \\
    &1+\frac{4}{N-4-2\sqrt{N-1}}
    &&\mbox{if} \,\,\, N \ge 11,
  \end{aligned}
  \right.
\end{equation*}
and satisfy $p_{\text{JL}}>p_{\text{S}}$ if $N \ge 3$.

The problems concerning \eqref{eq:1.3} and \eqref{eq:1.4} have been studied for decades.
First we recall some known results dealing with the subcritical case, namely 
\begin{equation*}
  \left\{
  \begin{aligned}
    &\limsup_{u \to \infty} \frac{f(u)}{u^p}<\infty 
    \,\,\, \mbox{for some} \,\,\, p<p_{\text{S}}
    \,\,\, \mbox{if} \,\,\, N \ge 3,
    \\
    &\limsup_{u \to \infty} \frac{f(u)}{e^{u^q}}<\infty 
    \,\,\, \mbox{for some} \,\,\, q<2
    \,\,\,\mbox{if} \,\,\, N=2.
  \end{aligned}
  \right.
\end{equation*}
Under this condition and some mild assumptions on $f$,
it was shown in \cites{CL,Fila,NST,Q} that 
\eqref{eq:1.3} holds for any global-in-time solution $u$.
The inequality \eqref{eq:1.4} was also shown to hold under the condition
\begin{equation*}
  \lim_{u \to \infty} \frac{f(u)}{u^p} \in (0,\infty)
  \quad \mbox{for some} \quad p \in (1,p_{\text{S}})
\end{equation*}
(see \cites{GK1,GMS2,Q}).
Furthermore, recently in \cites{HZ,Souplet},
it was shown that \eqref{eq:1.4} is valid for more general subcritical nonlinear terms
including $f(u)=u^p[\log (e+u)]^r$ ($p \in (1,p_{\text{S}})$, $r \in \mathbb{R}$).

Next we consider the supercritical case.
It is known that if $f$ satisfies either 
\begin{equation}
  \label{eq:1.5}
  \left\{
  \begin{aligned}
    &N \ge 3 \,\,\, \mbox{and} \,\,\,
    \lim_{u \to \infty} \frac{f(u)}{u^p} \in (0,\infty) 
    \,\,\, \mbox{for some} \,\,\, p \in (p_{\text{S}},p_{\text{JL}}),
    \,\,\, \mbox{or} \,\,\, 
    \\
    &3 \le N \le 9 \,\,\,\mbox{and} \,\,\,
    \lim_{u \to \infty} \frac{f(u)}{e^u} \in (0,\infty),
  \end{aligned}
  \right.
\end{equation}
then \eqref{eq:1.3} holds for any radially symmetric global-in-time solution $u$
(\cites{CFG,FPo,FP,GV}).
By results obtained in \cites{CdPM,dPMWZ2,DGL,GV,LT,NST,PV,Suzuki},
we see that the condition $p \in (1,p_{\text{S}}) \cup (p_{\text{S}},p_{\text{JL}})$ 
(resp. $N \le 9$) seems to be optimal for \eqref{eq:1.3}
when $f(u)=\lambda u^p$ (resp. $f(u)=\lambda e^u$), $\lambda>0$.
Indeed, it is shown that \eqref{eq:1.3} fails for some solution $u$ 
if one of the following holds:
\begin{itemize}
  \item[(i)]
    $N \ge 3$ and $f(u)=\lambda u^{p_{\text{S}}}$ (\cites{CdPM,dPMWZ2,GV,NST,Suzuki});
  \item[(ii)]
    $N \ge 11$, $k>0$, $\Omega=B_R$,
    $f(u)=\lambda u^p$ for some $p>p_{\text{JL}}$
    and $\lambda >0$ is chosen suitably (\cite{DGL});
  \item[(iii)]
    $N \ge 10$, $\Omega=B_R$,
    $f(u)=\lambda e^u$ and $\lambda>0$ is chosen suitably (\cites{DGL,LT,PV}).
\end{itemize}
Here $B_R:=\{ |x|<R\}$, $R>0$.
We remark that for the initial value problem
\begin{equation}
  \label{eq:1.6}
  \left\{
    \begin{array}{ll}
      \partial_t u = \Delta u + u^p, & x\in\mathbb{R}^N, \,\,\, t>0, 
      \\[3pt]
      u(x,0) = u_0(x), & x\in \mathbb{R}^N, 
    \end{array}
  \right.
\end{equation}
\eqref{eq:1.3} is known to fail also in the case $p=p_{\text{JL}}$ (\cite{PY}).
We furthermore remark that if one focuses on the case where $k=0$ and $f(u)=\lambda u^p$,
and considers a convex domain $\Omega$,
then \eqref{eq:1.3} is known to be true even when $p \ge p_{\text{JL}}$
(\cites{BS,CDZ,GV,Mizoguchi2,Souplet0}).

As for blow-up rates of solutions, 
it is known that \eqref{eq:1.4} is true under the condition \eqref{eq:1.5},
provided that $u$ is radially symmetric 
and that its maximum is located at the center of $\Omega=B_R$ (\cites{CFG,FP,MM}).
Moreover, there are results suggesting that
the condition $p \in (1,p_{\text{S}}) \cup (p_{\text{S}},p_{\text{JL}})$ 
is necessary for \eqref{eq:1.4} in the case $f(u)=u^p$.
Indeed, it was shown in \cites{dPMW,dPMWZ,dPMWZZ,Harada,Schweyer} that
there is a solution with type II blow-up if $f(u)=u^{p_{\text{S}}}$ and $3 \le N \le 6$,
where only the initial value problem \eqref{eq:1.6} 
is considered in \cites{dPMWZZ,Harada,Schweyer}.
In addition, 
for \eqref{eq:1.6},
the existence of a solution with type II blow-up is proved
also in the case $p \ge p_{\text{JL}}$ and $N \ge 11$
(\cite{HV,Mizoguchi,Seki}).
We remark that if a solution $u(x,t)$ is non-decreasing with respect to $t$,
then \eqref{eq:1.4} is known to be valid for general convex nonlinear teams (\cite{FM}).
We refer the reader to the book~\cite{QS}
which covers a wide range of related topics. 

The aim of this paper is to provide a general condition on $f$ 
such that \eqref{eq:1.3} and \eqref{eq:1.4} are true.
We will prove that, 
under some assumption on $f$ weaker than \eqref{eq:1.5},
both \eqref{eq:1.3} and \eqref{eq:1.4} hold 
for any radially symmetric solution of \eqref{eq:1.1}
whose maximum is located at the center of $\Omega=B_R$
(for details, see Theorems~\ref{theorem:1.1} and~\ref{theorem:1.2} below).
Our results can treat nonlinear terms such as 
$f(u)=u^p \log (e+u)$ ($p \in (p_{\text{S}},p_{\text{JL}})$)
and $f(u)=\exp (\exp (\exp (u)))$,
which cannot be handled in previous studies.
In particular, in the case $3 \le N \le 9$,
we can assert that both \eqref{eq:1.3} and \eqref{eq:1.4} are true
for nonlinear terms growing extremely faster than $e^u$.
Other concrete examples of nonlinear terms will be given in the next subsection.

\subsection{Main results}

We describe conditions to be assumed in our main results.
In this paper,
we only deal with radially symmetric solutions:
it will be assumed that
\begin{itemize}
  \item[(A1)]
    $\Omega=B_R=\{ |x|<R\}$ for some $R>0$ 
    and $u_0(x)=u_0(|x|)$.
\end{itemize}
For the nonlinear term $f$,
we will impose the following condition:
\begin{itemize}
  \item[(A2)]
    there is $f_0\in C^1([0,\infty))$ such that
    \begin{equation}
      \label{eq:1.7}
      \left\{
      \begin{aligned}
        &f_0(u)>0, \,\,\, 
        F_0(u) := \int_u^\infty \frac{1}{f_0(\eta)}\, d\eta<\infty 
        \,\,\, \mbox{for}\,\,\, u>0,
        \\
        &\mbox{the limit}\,\,\, q:=\lim_{u\to\infty} f_0'(u)F_0(u) \,\,\, \mbox{exists},
        \\
        &\lim_{u\to\infty} \frac{f(u)}{f_0(u)} = 1.
        \end{aligned}
      \right.
    \end{equation}
\end{itemize}
We note that the limit $q$ in \eqref{eq:1.7} must satisfy $q \ge 1$ (see \cite{FI}).
We also note that $q$ stands for a number 
related to the H\"{o}lder conjugate of the growth rate of $f_0$.
Indeed, if $f_0(u)=u^p$ with $p>1$ (resp. $f_0(u)=e^u$), 
then $q$ coincides with $p/(p-1)$ (resp. $1$). 
From this fact, 
we can naturally generalize the condition \eqref{eq:1.5} by means of $q$ as follows:
\begin{itemize}
  \item[(A3)]
    $N \ge 3$ and $q_{\text{JL}}<q<q_{\text{S}}$,
    or $3 \le N \le 9$ and $q=1$.
\end{itemize}
Here the exponents $q_{\text{S}}$ and $q_{\text{JL}}$ 
denote the H\"{o}lder conjugates of $p_{\text{S}}$ and $p_{\text{JL}}$, respectively:
they are given by
\begin{equation*}
  q_{\text{S}}:=
  \left\{
  \begin{aligned}
    &1
    &&\mbox{if} \,\,\, N \le 2,
    \\
    &\frac{N+2}{4}
    &&\mbox{if} \,\,\, N \ge 3,
  \end{aligned}
  \right.
  \qquad
  q_{\text{JL}}:=
  \left\{
  \begin{aligned}
    &1
    &&\mbox{if} \,\,\, N \le 10,
    \\
    &\frac{N-2\sqrt{N-1}}{4}
    &&\mbox{if} \,\,\, N \ge 11.
  \end{aligned}
  \right.
\end{equation*}
Nonlinear terms satisfying (A2) and (A3) contain various functions 
which violate \eqref{eq:1.5}.
As examples, 
let us consider functions given by
\begin{equation*}
  f_1(u)=(u^p +u e^{\sin u}) [\log (e+u)]^{r_1},
  \quad
  f_2(u)=\exp (u^{r_2}) + u^{r_3} \cos^2 u,
  \quad
  f_3(u)=\exp^n (u),
\end{equation*}
where $p>1$, $r_1 \in \mathbb{R}$, $r_2,r_3>0$, $n \in \mathbb{N}$
and $\exp^n=\underbrace{\exp \circ \cdots \circ \exp}_{n}$.
Then obviously \eqref{eq:1.5} is not satisfied for $f=f_j$, $j=1,2,3$,
unless $r_1=0$, $r_2=1$, $n=1$.
On the other hand,
if we choose $f_{0,j}$ for $j=1,2,3$ as
\begin{equation*}
  f_{0,1}(u)=u^p [\log (e+u)]^{r_1},
  \qquad
  f_{0,2}(u)=\exp (u^{r_2}),
  \qquad
  f_{0,3}(u)=f_3(u)=\exp^n (u),
\end{equation*}
then one can check that \eqref{eq:1.7} is satisfied for $f=f_j$ and $f_0=f_{0,j}$
with $q$ given by 
\begin{equation*}
  q=\left\{
  \begin{aligned}
    &\frac{p}{p-1}
    &&\mbox{if} \,\,\, f_0=f_{0,1},
    \\
    &1
    &&\mbox{if} \,\,\, f_0=f_{0,2},f_{0,3}.
  \end{aligned}
  \right.
\end{equation*}
Complicated nonlinear terms including these examples
can be handled in a unified way by our analysis.

Let us state our main results. 
First result provides the uniform boundedness of global solutions. 

\begin{theorem}
  \label{theorem:1.1}
  Assume \textup{(A1)}--\textup{(A3)}.
  If a solution $u$ of \eqref{eq:1.1} exists globally,
  then
  \begin{equation}
    \notag 
    \limsup\limits_{t\to \infty} \|u(\cdot,t)\|_{L^\infty(\Omega)} < \infty. 
  \end{equation} 
\end{theorem}

The second result shows the non-existence of solutions with type II blow-up.

\begin{theorem}
  \label{theorem:1.2}
  Assume that \textup{(A1)}--\textup{(A3)} hold,
  and that a solution $u$ of \eqref{eq:1.1} blows up at $t=T<\infty$.
  Assume in addition that there exists $T_0\in [0,T)$ such that 
  \begin{equation}
    \label{eq:1.8}
    \|u(\cdot,t)\|_{L^\infty(\Omega)}=u(0,t) 
    \quad\mbox{for}\quad
    t\in [T_0,T). 
  \end{equation}
  Then 
  \begin{equation} 
    \notag 
    \liminf\limits_{t\to T} \frac{F(\|u(\cdot,t)\|_{L^\infty(\Omega)})}{T-t}>0.
  \end{equation} 
\end{theorem}

\begin{remark}
  \label{remark:1.1}
  $(\mathrm{i})$ 
  The assumption \eqref{eq:1.8} is satisfied if $u_0$ is non-increasing in $|x|$.
  This assumption is known to be not needed 
  in the case $f(u)=u^p$ \textup{(}see \cite{MM}\textup{)}.
  

  \smallskip 

  \noindent 
  $(\mathrm{ii})$ 
  If $f(u)=u^p [\log (e+u)]^r$ \textup{(}$p>1$, $r \in \mathbb{R}$\textup{)},
  then \eqref{eq:1.4} is written as
  \begin{equation*} 
    \limsup\limits_{t\to T} \, (T-t)^{\frac{1}{p-1}} 
    \left(
      \log\frac{1}{T-t}
    \right)^{\frac{r}{p-1}}
    \|u(\cdot,t)\|_{L^\infty(\Omega)} <\infty.
  \end{equation*} 
  This is shown in \cites{HZ,Souplet} in the subcritical case $p \in (1,p_{\text{S}})$.
  Theorem~\textup{\ref{theorem:1.2}} shows that this is still valid 
  for $p \in (p_{\text{S}},p_{\text{JL}})$,
  provided that $u$ is radially symmetric and satisfies \eqref{eq:1.8}.
\end{remark}

We mention the strategy of the proofs of Theorems~\ref{theorem:1.1} and \ref{theorem:1.2}.
Both of the theorems are proved by applying an argument employed in \cite{CFG,FP}.
At the first step,
we observe that 
the intersection number of a radially symmetric singular steady state $u^*(x)$
and a radially symmetric solution $u(x,t)$ of \eqref{eq:1.1} 
is bounded by a constant independent of $t$.
Then at the second step,
we derive bounds for a rescaled function 
of the form $w^t_\lambda(y,\tau)=H(u(\lambda y,t+\lambda^2 \tau),\lambda)$,
where $\lambda>0$ is a parameter
and $H(u,\lambda)$ is a function chosen appropriately.
Finally at the third step,
we prove \eqref{eq:1.3} and \eqref{eq:1.4} by contradiction.
Assuming that \eqref{eq:1.3} or \eqref{eq:1.4} is false,
we first show that some sequence 
$\{w^{t_i}_{\lambda_i}\}_{i=1}^\infty$ of the rescaled functions 
converges to a steady state $\phi(y)$ of a limiting equation in the whole space.
We then observe that a resclaed function corresponding to $u^*$,
namely $H(u^*(\lambda_i y),\lambda_i)$,
converges to a singular steady state $\phi^*(y)$ of the limiting equation.
If $\phi$ and $\phi^*$ intersect infinitely many times,
which is indeed true under the condition (A3),
then it can be concluded that
the intersection number of $u^*(x)$ and $u(x,t_i)$ diverges to infinity as $i \to \infty$,
contrary to the fact derived in the first step.

For the nonlinear terms $f(u)=u^p$ and $f(u)=e^u$,
there are well-known similar transformations 
such that the equation in \eqref{eq:1.1} is invariant under them.
The third step of the proofs is then successfully done 
by using such transformations
in defining the rescaled function $w^t_\lambda$. 
More precisely,
if $u$ satisfies $\partial_t u=\Delta u +u^p$ (resp. $\partial_t u=\Delta u +e^u$) 
and if $w^t_\lambda$ is defined by
\begin{equation}
  \label{eq:1.9}
  w^t_\lambda(y,\tau)
  =\lambda^{\frac{2}{p-1}} u(\lambda y,t+\lambda^2 \tau)
  \qquad
  \left(
    \mbox{resp.} \,\,\,
    w^t_\lambda(y,\tau)
    =u(\lambda y,t+\lambda^2 \tau) +2\log \lambda
  \right),
\end{equation}
then $u$ and $w^t_\lambda$ satisfy the same semilinear heat equation;
hence the limiting equation in the third step also must be the same.
The above transformations still work 
for nonlinear terms satisfying \eqref{eq:1.5},
as demonstrated in \cite{CFG,FP}.
However, the argument with these transformations would break down
if \eqref{eq:1.5} is violated.
In order to apply the argument for general nonlinear terms, 
we utilize the following quasi similar transformations instead:  
\begin{equation}
  \label{eq:1.10}
  w^t_\lambda(y,\tau) 
  =\left\{
  \begin{aligned}
    &(q-1)^{q-1} \lambda^{2(q-1)}
    F_0 \left(
      u(\lambda y,t_0+\lambda^2 \tau)
    \right)^{-(q-1)}
    &&\mbox{if} \quad q>1,
    \\
    &-\log 
    \left(
      F_0 \left(
        u(\lambda y,t_0+\lambda^2 \tau)
      \right)
    \right)
    +2\log \lambda
    &&\mbox{if} \quad q=1.
  \end{aligned}
  \right.
\end{equation}
Here $F_0$ is the function given in \eqref{eq:1.7}.
This type of transformation was introduced in \cite{Fujishima} 
to examine the blow-up set of solutions, 
and has been used to analyze non-scale-invariant equations 
(see for instance \cites{FI,Miyamoto}).
One can easily check that if $q>1$ and $f(u)=f_0(u)=u^p$ with $p=q/(q-1)$,
or if $q=1$ and $f(u)=f_0(u)=e^u$,
then the rescaled functions defined by \eqref{eq:1.10}
coincide with the ones defined by \eqref{eq:1.9}.
One can also check that by \eqref{eq:1.10},
the equation in \eqref{eq:1.1} is transformed into
\begin{equation}
  \label{eq:1.11}
  \begin{aligned}
    &\partial_\tau w^t_\lambda -\Delta_y w^t_\lambda -(w^t_\lambda)^{\frac{q}{q-1}}
    =\left(
        \frac{f(u)}{f_0(u)}-1
    \right)
    (w^t_\lambda)^{\frac{q}{q-1}} 
    +\frac{(f_0'(u)F_0(u)-q) |\nabla_y w^t_\lambda|^2}{(q-1)w^t_\lambda}
    &&\mbox{if}\,\,\, q>1,
        \\
    &\partial_\tau w^t_\lambda -\Delta_y w^t_\lambda -e^{w^t_\lambda}
    =\left(
        \frac{f(u)}{f_0(u)}-1
    \right)
    e^{w^t_\lambda} + (f_0'(u)F_0(u)-1)|\nabla_y w^t_\lambda|^2
    &&\mbox{if}\,\,\, q=1
  \end{aligned}
\end{equation}
(for the derivation of \eqref{eq:1.11},
see \cite{FI}*{Proposition~3.1} or Section~\ref{section:5}).
Although extra terms, which are due to the lack of scale invariance, 
appear on the right-hand side of \eqref{eq:1.11},
they disappear as $u$ diverges to infinity by the assumption (A2).
Therefore the limiting equation again becomes
$\partial_\tau w=\Delta_y w +w^p$ or $\partial_\tau w=\Delta_y w +e^w$,
and the third step of the proofs would be done.
In order to justify the convergence of the extra terms to $0$,
we need to derive appropriate bounds for $w^t_\lambda$ and $|\nabla_y w^t_\lambda|$.
This is one of difficulties in this study,
and will be discussed in Section~\ref{section:5}. 

This paper is organized as follows.
In Section~\ref{section:2}, 
we introduce a key proposition
from which Theorems~\ref{theorem:1.1} and \ref{theorem:1.2} follow easily.
The remaining sections discuss the proof of the key proposition.
In Section~\ref{section:3},
we construct a radially symmetric singular steady state of \eqref{eq:1.1}.
In Section~\ref{section:4}, 
we recall some known facts on a pointwise estimate for solutions of \eqref{eq:1.1}
and the number of zeros for solutions of elliptic and parabolic equations.
The proof of the key proposition is given in Section~\ref{section:5}. 

\section{Proofs of Theorems~\ref{theorem:1.1} and \ref{theorem:1.2}}
\label{section:2}

This section provides proofs of Theorems~\ref{theorem:1.1} and \ref{theorem:1.2}. 
The following proposition is the key to proving the theorems.
Its proof will be given in subsequent sections. 

\begin{proposition}
  \label{proposition:2.1}
  Assume \textup{(A1)}--\textup{(A3)} and \eqref{eq:1.8}.
  Let $u$ be a $($radially symmetric$)$ solution of \eqref{eq:1.1}
  and $T\in (0,\infty]$ its maximal existence time.
  Then one of the following holds:
  \begin{gather}
    \label{eq:2.1}
    \liminf_{t\to T} \sup_{s \in (t,T)} 
    \frac{F(\|u(\cdot,t)\|_{L^\infty(\Omega)})-F(\|u(\cdot,s)\|_{L^\infty(\Omega)})}{s-t} 
    >0 \, ;
    \\
    \label{eq:2.2}
    \limsup_{t\to T} \|u(\cdot,t)\|_{L^\infty(\Omega)}<\infty. 
  \end{gather}
\end{proposition}

For the proof of Theorem~\ref{theorem:1.1},
we use a monotonicity property of radially symmetric solutions of \eqref{eq:1.1} 
stated in \cite[Corollary~$2$]{NS}. 

\begin{proposition}[\cite{NS}]
  \label{proposition:2.2}
  Assume \textup{(A1)} and 
  write $u(x,t)=U(|x|,t)$ for a radially symmetric solution of \eqref{eq:1.1}.
  If $u$ exists globally in time,
  then there exists $T_0>0$ such that 
  $\partial_r U(r,t)<0$ for $(r,t) \in (0,R] \times [T_0,\infty)$.
  In particular, 
  \eqref{eq:1.8} holds for $T=\infty$.
\end{proposition}

Before proceeding to the proofs of Theorems~\ref{theorem:1.1} and \ref{theorem:1.2}, 
we show the following lemma.

\begin{lemma}
  \label{lemma:2.1}
  Let $T$ be the maximal existence time of a solution $u$ of \eqref{eq:1.1}. 
  If \eqref{eq:2.1} holds,
  then 
  \begin{equation}
    \notag
    T<\infty
    \qquad\mbox{and}\qquad
    \liminf\limits_{t\to T} \frac{F(\|u(\cdot,t)\|_{L^\infty(\Omega)})}{T-t}>0.
  \end{equation}
\end{lemma}

\begin{proof}
  Put $\phi(t) :=F(\| u(\cdot,t) \|_{L^\infty(\Omega)})$.
  By the assumption \eqref{eq:2.1}, 
  we can take $c>0$ and $t_* \in (0,T)$ such that 
  \begin{equation}
    \label{eq:2.3}
    \sup_{s \in (t,T)} 
    \frac{\phi(t)-\phi(s)}{s-t}>c
    \quad\mbox{for}\quad
    t\in [t_*,T).
  \end{equation}
  It is then sufficient to prove that
  \begin{equation}
    \label{eq:2.4}
    \phi(t) \ge c(T-t)
    \quad\mbox{for}\quad
    t\in [t_*,T).
  \end{equation}  
  
  For $t \in [t_*,T)$,
  we set 
  \begin{equation*}
    A:=\{ s \in [t,T): \, \phi (t)-\phi(s) \ge c(s-t) \}.
  \end{equation*}
  Then $A$ is not empty, 
  since $t$ is clearly contained in $A$. 
  Hence we can define $\overline{T}$ by
  \begin{equation*}
    \overline{T}:=\sup A \in [t,T].
  \end{equation*}
  By definition,
  there is a sequence $\{t_i\} \subset [t,\overline{T})$ such that 
  \begin{equation}
    \label{eq:2.5}
    \lim_{i \to \infty} t_i=\overline{T},
    \qquad
    \phi (t)-\phi(t_i) \ge c(t_i-t).
  \end{equation}
  Now let us show that 
  \begin{equation}
    \label{eq:2.6}
    \overline{T}=T.
  \end{equation}
  On the contrary, 
  suppose that $\overline{T}<T$.
  Then by \eqref{eq:2.5} we have
  \begin{equation*}
    \phi (t)-\phi(\overline{T}) \ge c(\overline{T}-t).
  \end{equation*}
  Using \eqref{eq:2.3} with $t=\overline{T}$,
  we also have
  \begin{equation*}
    \phi (\overline{T})-\phi(s_0) \ge c(s_0-\overline{T})
    \quad \mbox{for some} \quad
    s_0 \in (\overline{T},T).
  \end{equation*}
  Combining the above inequalities,
  we deduce that
  \begin{equation*}
    \phi (t)-\phi(s_0) \ge c(s_0-t).
  \end{equation*}
  This shows that $s_0$ is an element of $A$ which exceeds $\overline{T}=\sup A$,
  a contradiction.
  Therefore \eqref{eq:2.6} is verified.
  
  From \eqref{eq:2.5}, \eqref{eq:2.6} and the fact that $\phi>0$,
  we obtain
  \begin{equation*}
    \phi (t) 
    \ge \limsup_{i\to\infty} \left\{ \phi (t_i) +c(t_i-t) \right\}
    \ge \lim_{i\to\infty} c(t_i-t)
    =c(\overline{T}-t)
    =c(T-t).
  \end{equation*}
  We have thus derived \eqref{eq:2.4},
  and the lemma follows. 
\end{proof}

\begin{remark}
  \label{remark:2.1}
  The same assertions in Lemma~\textup{\ref{lemma:2.1}} hold
  if \eqref{eq:2.1} is replaced with the condition
  \begin{equation}
    \label{eq:2.7}
    \liminf_{t\to T} \frac{M'(t)}{f(M(t))}> 0,
  \end{equation}
  where $M(t)=\|u(\cdot,t)\|_{L^\infty(\Omega)}$.
  Indeed, this condition leads to \eqref{eq:2.1},
  since
  \begin{equation*}
    \sup_{s \in (t,T)} 
    \frac{F(M(t))-F(M(s))}{s-t}
    \ge
    \lim_{s \to t} 
    \frac{F(M(t))-F(M(s))}{s-t}
    =-\frac{d}{dt} F(M(t))
    =\frac{M'(t)}{f(M(t))}.
  \end{equation*}
  Although the condition \eqref{eq:2.7} is often considered in previous studies,
  we adopt \eqref{eq:2.1},
  because an estimate of the derivative with respect $t$ is required for \eqref{eq:2.7},
  while it is not for \eqref{eq:2.1}.
  This is crucial in the analysis of the transformed equation \eqref{eq:1.11}.
  In a regularity theory for parabolic equations,
  bounds of the H\"{o}lder norms of lower order terms are needed
  to obtain a pointwise estimate for the derivative of a solution with respect the time variable.
  However, obtaining such bounds in \eqref{eq:1.11} requires not only H\"{o}lder estimates
  for the gradient term involved on the right-hand side of \eqref{eq:1.11},
  but also additional assumptions stronger than \textup{(A2)}. 
  By considering \eqref{eq:2.1} instead of \eqref{eq:2.7},
  these difficulties can be avoided.
  \end{remark}

Let us now prove the theorems. 

\begin{proof}[Proof of Theorem~\ref{theorem:1.1}]
  Let $u$ be a global solution of \eqref{eq:1.1}. 
  Then Proposition~\ref{proposition:2.1} can be applied,
  because Proposition~\ref{proposition:2.2} gives $\|u(\cdot,t)\|_{L^\infty(\Omega)}=u(0,t)$ for large $t$.
  Since it is assumed that $u$ exists globally, 
  we see from the contraposition of  Lemma~\ref{lemma:2.1} that \eqref{eq:2.1} fails. 
  Therefore \eqref{eq:2.2}, which is our desired conclusion, is true. 
\end{proof}

\begin{proof}[Proof of Theorem~\ref{theorem:1.2}]
  Let $u$ be a solution of \eqref{eq:1.1} which blows up at $t=T<\infty$. 
  By Proposition~\ref{proposition:2.1},  
  we see that \eqref{eq:2.1} or \eqref{eq:2.2} holds.
  The condition \eqref{eq:2.2}, however, is not satisfied since $u$ blows up at $t=T$,
  and hence \eqref{eq:2.1} must be true.
  The desired conclusion is thus obtained by applying Lemma~\ref{lemma:2.1}. 
\end{proof}

The remaining sections are devoted to the proof of Proposition~\ref{proposition:2.1}.

\section{Existence of a singular steady state}
\label{section:3}

Proposition~\ref{proposition:2.1} is proved 
by investigating the intersection number of a singular steady state and a solution of \eqref{eq:1.1}.
As a preliminary step,
in this section we construct a radially symmetric singular steady state;
this is done by finding a singular solution of the equation
\begin{equation}
  \label{eq:3.1}
  U''+\frac{N-1}{r}U'+f(U)=0.
\end{equation}
It is well-known that if $f(u)=u^p$ and $p>N/(N-2)$ (resp. $f(u)=e^u$),
this equation has an explicit singular solution $\Phi^*_p$ (resp. $\Phi^*_\infty$) given by 
\begin{equation}
  \label{eq:3.2}
  \Phi_p^*(r) := 
  \left[
    (p-1)\left(
      2N-\frac{4p}{p-1}
    \right)^{-1}
  r^2
  \right]^{-1/(p-1)}
  \quad 
  \left( 
  \mbox{resp.}\,\,\, \Phi_\infty^*(r) =-\log \frac{r^2}{2N-4}
  \right).
\end{equation}
We note that $\Phi_p^*$ is written as $\Phi_p^*(r)=F^{-1}((2N-4q)^{-1} r^2)$,
where $q=p/(p-1)$ if $f(u)=u^p$ and $q=1$ if $f(u)=e^u$. 
The goal of this section is to show that for $f$ satisfying (A2),
the equation \eqref{eq:3.1} admits a singular solution 
behaving similarly to $F^{-1}((2N-4q)^{-1} r^2)$ as $r \to 0$.

\begin{proposition}
  \label{proposition:3.1}
  Let $N\ge 3$ and assume that \textup{(A2)} holds with $q<q_{\mathrm{S}}$. 
  Then there is $r_0>0$ such that the problem~\eqref{eq:3.1} possesses a solution 
  $U^*\in C^2(0,r_0)$ which is singular at the origin and is of the form 
  \begin{equation}
    \label{eq:3.3}
    U^*(r) = F_0^{-1}
    \left(
      (2N-4q)^{-1} r^2 (1+\theta(r))
    \right), 
  \end{equation}
  where $\theta\in C^2(0,r_0)$ is a function satisfying $\theta(r)\to 0$ as $r\to 0$. 
\end{proposition}

The same result was proved in \cite{Miyamoto} under the assumptions that
\begin{equation*}
  f \in C^2(u_0,\infty) 
  \quad\mbox{for some}\quad 
  u_0>0,
  \qquad
  \frac{f(u)f''(u)}{f'(u)^2} \to \frac{1}{q}
  \quad\mbox{as}\quad
  u \to \infty.
\end{equation*}
If these are satisfied,
then \eqref{eq:1.7} holds for $f_0=f$,
because L'Hospital's rule gives
\begin{equation*}
  \lim_{u \to \infty} \frac{1}{f'(u)F(u)}
  =\lim_{u \to \infty} \frac{(1/f'(u))'}{F'(u)}
  =\lim_{u \to \infty} \frac{-f''(u)/f'(u)^2}{-1/f(u)}
  =\lim_{u \to \infty} \frac{f(u)f''(u)}{f'(u)^2}
  =\frac{1}{q}.
\end{equation*}
Our claim in Proposition~\ref{proposition:3.1} is that
the existence of a singular solution $U^*$ is guaranteed 
under the weaker assumption (A2).

We set up notation for function spaces to be used in the proof of Proposition~\ref{proposition:3.1}.
For $k \in \mathbb{N}$ and an interval $I \subset \mathbb{R}$,
let $BC^k(I)$ denote the set of all $C^k$ functions on $I$ whose derivatives up to $k$ are bounded.
This set is a Banach space with the norm defined by
\begin{equation*}
  \| X\|_{BC^k(I)}:=\sup_{s \in I} |X(s)| +\sum_{j=1}^k \sup_{s \in I} 
  \left|
  \frac{d^k X}{ds^k}(s)
  \right|.
\end{equation*}
For $a \in \mathbb{R}$, $b>0$ 
and a non-negative bounded function $m=m(s)$ defined on $(-\infty,a]$,
we write
\begin{gather*}
  \mathcal{X}_{a,m}:=
  \left\{
  X \in BC^1((-\infty,a]): \, |X(s)|+|X'(s)| \le m(s)
  \right\},
  \\
  \mathcal{Y}_{a,b,m}:=
  \left\{
  X \in BC^2((-\infty,a]): \, |X(s)|+|X'(s)| \le m(s), \, |X''(s)| \le b
  \right\}.
\end{gather*}
We think of $\mathcal{X}_{a,m}$ as a metric space 
with the metric induced by the norm $\| \cdot \|_{BC^1((-\infty,a])}$.

To prove Proposition~\ref{proposition:3.1},
we examine properties of the function spaces defined above.
\begin{lemma}
  \label{lemma:3.1}
  The following assertions hold:
  \begin{itemize}
    \item[$\mathrm{(i)}$] $\mathcal{X}_{a,m}$ is a closed convex subset of $BC^1((-\infty,a])$; 
    \item[$\mathrm{(ii)}$] if $\lim\limits_{s \to -\infty} m(s)=0$,
    then $\mathcal{Y}_{a,b,m}$ is relatively compact in $\mathcal{X}_{a,m}$.
  \end{itemize}  
\end{lemma}

\begin{proof}
  The proof of (i) is straightforward.
  We only give a proof of (ii).

  Suppose that $\lim\limits_{s \to -\infty} m(s)=0$,
  and let $\{ X_n\}$ be any sequence in $\mathcal{Y}_{a,b,m}$.
  Since $\mathcal{X}_{a,m}$ is closed in $BC^1((-\infty,a])$,
  the assertion (ii) follows if we show that 
  $\{ X_n\}$ contains a convergent subsequence in $BC^1((-\infty,a])$.
  For each bounded closed interval $I$ in $(-\infty,a]$,
  there is a subsequence of $\{ X_n\}$ which is convergent in $BC^1(I)$,
  because the definition of $\mathcal{Y}_{a,b,m}$ shows that 
  $\{ X_n\}$ is bounded in $BC^2(I)$ which is compactly embedded in $BC^1(I)$.
  From this fact and a diagonal argument, 
  we can take a subsequence $\{ X_{n_j}\}$ of $\{ X_n\}$ and $X \in C^1((-\infty,a])$ such that
  \begin{equation}
    \label{eq:3.4}
    \|X_{n_j}-X\|_{BC^1([s_0,a])} \to 0
    \quad\mbox{for any}\,\,\, s_0 \in (-\infty,a) \,\,\, \mbox{as} \,\,\, j \to \infty.
  \end{equation}
  We then have $X \in \mathcal{X}_{a,m}$,
  since letting $j \to \infty$ in the inequality $|X_{n_j}(s)|+|X_{n_j}'(s)| \le m(s)$
  gives $|X(s)|+|X'(s)| \le m(s)$.
  
  For $s_0 \in (-\infty,a]$,
  the norm of $X_{n_j}-X$ is estimated as
  \begin{align*}
    \|X_{n_j}-X\|_{BC^1((-\infty,a])} 
    &\le \|X_{n_j}-X\|_{BC^1([s_0,a])} +\|X_{n_j}-X\|_{BC^1((-\infty,s_0])}
    \\
    &\le \|X_{n_j}-X\|_{BC^1([s_0,a])} +2\sup_{s \in (-\infty,s_0]} m(s).
  \end{align*}
  Hence by \eqref{eq:3.4} we find that 
  \begin{equation*}
    \limsup\limits_{j \to \infty} \|X_{n_j}-X\|_{BC^1((-\infty,a])} 
    \le 2\sup_{s \in (-\infty,s_0]} m(s).
  \end{equation*}
  Since the right-hand side converges to $0$ as $s_0 \to -\infty$,
  we conclude that 
  \begin{equation*}
    X_{n_j} \to X
    \quad\mbox{in} \,\,\, BC^1((-\infty,a]) \,\,\, \mbox{as} \,\,\, j \to \infty.
  \end{equation*}
  Thus the lemma follows.
\end{proof}

We are ready to prove Proposition~\ref{proposition:3.1}.

\begin{proof}[Proof of Proposition~\ref{proposition:3.1}]
  First we consider the change of variables
  \begin{equation*}
    U(r)=F_0^{-1}
    \left(
    (2N-4q)^{-1} e^{2s-X(s)}
    \right),
    \qquad
    s=\log r,
  \end{equation*}
  and transform \eqref{eq:3.1} into an equation for $X(s)$.
  By definition,
  \begin{equation}
    \label{eq:3.5}
    (2N-4q)F_0(U(r))= e^{2s-X(s)}.
  \end{equation}
  Differentiating this with respect to $r$,
  we have
  \begin{equation}
    \label{eq:3.6}
    (2N-4q)\frac{U'(r)}{f_0(U(r))}
    =e^{s-X(s)}(X'(s)-2).
  \end{equation}
  Squaring the both sides of this equality and using \eqref{eq:3.5}, 
  we also have
  \begin{equation}
    \label{eq:3.7}
    (2N-4q)\frac{U'(r)^2}{f_0(U(r))^2}
    =F_0(U(r)) e^{-X(s)}(X'(s)-2)^2.
  \end{equation}
  We further differentiate \eqref{eq:3.6} to obtain
  \begin{equation*}
    (2N-4q)
    \left(
    \frac{U''(r)}{f_0(U(r))} -\frac{f_0'(U(r))U'(r)^2}{f_0(U(r))^2}
    \right)
    =e^{-X(s)} 
    \left\{ 
    X''(s) -(X'(s)-2)(X'(s)-1)
    \right\}.
  \end{equation*}
  This together with \eqref{eq:3.7} shows that
  \begin{multline}
   \label{eq:3.8}
   (2N-4q) \frac{U''(r)}{f_0(U(r))} 
   \\
   =e^{-X(s)} 
   \left\{ 
   X''(s) -(X'(s)-2)(X'(s)-1) +f_0'(U(r)) F_0(U(r)) (X'(s)-2)^2
   \right\}.
  \end{multline}
  Plugging \eqref{eq:3.6} and \eqref{eq:3.8} into \eqref{eq:3.1},
  we arrive at
  \begin{equation}
   \label{eq:3.9}
    X'' -(X'-2)(X'-N) 
    +f_0'(\zeta(X,s)) F_0(\zeta(X,s)) (X'-2)^2
    +(2N-4q) \frac{f(\zeta(X,s))}{f_0(\zeta(X,s))} e^{X}=0,
  \end{equation}
  where
  \begin{equation*}
    \zeta(X,s):=F_0^{-1}
    \left(
    (2N-4q)^{-1} e^{2s-X}
    \right).
  \end{equation*}  
  If we construct a solution of \eqref{eq:3.9} satisfying 
  $X\in C^2((-\infty,a])$ for some $a\in \mathbb{R}$ and
  \begin{equation}
    \label{eq:3.10} 
    \lim_{s \to -\infty} X(s)=0,
  \end{equation}
  then a singular solution of the form \eqref{eq:3.3}
  is obtained by setting $\theta (r)=e^{-X(s)}-1$, $s=\log r$.
  
  Let us derive an integral equation corresponding to \eqref{eq:3.9}.
  Notice that \eqref{eq:3.9} is written as
  \begin{equation}
   \label{eq:3.11}
    X'' +(N+2-4q)X' +(2N-4q) X +h_1(X,X') +h_2(X,X',s)=0,
  \end{equation}  
  where
  \begin{gather*}
    h_1(X,Y):=(2N-4q)(e^X-1-X) +(q-1)Y^2,
    \\
    h_2(X,Y,s):=
    (2N-4q) 
    \left( 
    \frac{f(\zeta(X,s))}{f_0(\zeta(X,s))}-1
    \right) 
    e^{X} 
    +(f_0'(\zeta(X,s)) F_0(\zeta(X,s))-q) (Y-2)^2.
  \end{gather*} 
  By the assumption $q<q_{\text{S}}$,
  we see that the real part of every root of the quadratic polynomial
  $\lambda^2+(N+2-4q)\lambda+2N-4q$ is negative.
  This shows that the solution $Z$ of the initial value problem 
  \begin{equation}
    \label{eq:3.12}
    \left\{
    \begin{aligned}
      &Z''(s)+(N+2-4q)Z'(s)+(2N-4q)Z(s)=0,
      \quad s \in \mathbb{R},
      \\
      &Z(0)=0,\quad Z'(0)=1, 
    \end{aligned}
    \right.
  \end{equation}
  satisfies
  \begin{equation}
    \label{eq:3.13}
    |Z(s)|+|Z'(s)|+|Z''(s)| \le \beta e^{-\alpha s}
    \quad\mbox{for}\quad s \ge 0,
  \end{equation}
  where $\alpha$ and $\beta$ are positive constants.  
  Keeping this fact in mind,
  we convert \eqref{eq:3.11} into the integral equation
  \begin{equation}
    \label{eq:3.14}
    X(s)=\Psi[X](s):=-\int_{-\infty}^s 
    \left(
    h_1(X(\eta),X'(\eta)) +h_2(X(\eta),X'(\eta),\eta) 
    \right) 
    Z(s-\eta) \, d\eta.
  \end{equation}
  By \eqref{eq:3.12} and \eqref{eq:3.13},
  one can easily verify that if $X$ satisfies 
  $X \in BC^1((-\infty,\tilde a])$ and \eqref{eq:3.14} for some $\tilde a \in \mathbb{R}$,
  then $X$ is of class $C^2$ and satisfies \eqref{eq:3.11} on $(-\infty,\tilde a]$.
  The task is thus to find a fixed point $X$ of $\Psi$ satisfying \eqref{eq:3.10}.

  We set up a function space where a fixed point will be obtained.
  Observe that $h_1(X,Y)=O(X^2+Y^2)$ as $X,Y \to 0$.
  Hence there is $\delta_0>0$ such that
  \begin{equation}
    \label{eq:3.15}
    |h_1(X,Y)| \le \frac{\alpha}{4\beta} (|X|+|Y|)
    \quad\mbox{if}\quad |X|+|Y| \le \delta_0,
  \end{equation}
  where $\alpha$ and $\beta$ are the constants given in \eqref{eq:3.13}.  
  Since $F_0^{-1}(\eta)\to \infty$ as $\eta \to 0$,
  we see that $\zeta(X,s) \to \infty$ 
  locally uniformly for $X \in \mathbb{R}$ as $s \to -\infty$.
  This with \eqref{eq:1.7} shows that $h_2(X,Y,s) \to 0$ 
  locally uniformly for $(X,Y) \in \mathbb{R}^2$ as $s \to -\infty$.
  Hence
  \begin{equation}
    \label{eq:3.16}
    m(s):=\frac{4\beta}{\alpha}
    \sup \{ |h_2(X,Y,\tilde s)| :\, |X|+|Y| \le \delta_0, \, \tilde s \in (-\infty,s] \} \to 0
    \quad\mbox{as}\quad s \to -\infty.
  \end{equation}    
  In particular, there is $a \in (-\infty,0]$ such that
  \begin{equation}
    \label{eq:3.17}
    m(s) \le \delta_0
    \quad\mbox{for}\quad s \in (-\infty,a].
  \end{equation}
  For $a$ and $m$ defined in this way,
  we will show that $\Psi$ has a fixed point in $\mathcal{X}_{a,m}$
  by applying the Schauder fixed point theorem.
  
  Let $X \in \mathcal{X}_{a,m}$.
  Then we particularly have $|X|+|X'| \le \delta_0$ 
  on $(-\infty,a]$ by \eqref{eq:3.17}.
  Hence \eqref{eq:3.15} and the definition of $m$ show that 
  if $\eta \le s \le a$, then
  \begin{gather*}
    |h_1(X(\eta),X'(\eta))| \le \frac{\alpha}{4\beta} (|X(\eta)|+|X'(\eta)|) 
    \le \frac{\alpha}{4\beta} m(\eta) \le \frac{\alpha}{4\beta} m(s),
    \\
    |h_2(X(\eta),X'(\eta),\eta)| 
    \le \frac{\alpha}{4\beta} m(\eta) \le \frac{\alpha}{4\beta} m(s).
  \end{gather*}
  These together with \eqref{eq:3.13} yield
  \begin{equation}
    \label{eq:3.18}
    |\Psi[X](s)| 
    \le \int_{-\infty}^s 
    \left( 
    \frac{\alpha}{4\beta}m(s) +\frac{\alpha}{4\beta}m(s) 
    \right)
    \beta e^{-\alpha (s-\eta)} \, d\eta
    =\frac{1}{2} m(s).
  \end{equation}
  In a similar way,
  we deduce that
  \begin{equation}
    \label{eq:3.19}
    |(\Psi[X])'(s)|
    =\left|
    \int_{-\infty}^s 
    \left(
    h_1(X(\eta),X'(\eta)) +h_2(X(\eta),X'(\eta),\eta) 
    \right)
    Z'(s-\eta) \, d\eta
    \right|
    \le \frac{1}{2} m(s)
  \end{equation}
  and
  \begin{equation}
    \label{eq:3.20}
    \begin{aligned}
      |(\Psi[X])''(s)|
      &=\bigg|
      h_1(X(s),X'(s)) +h_2(X(s),X'(s),s) 
      \\
      &\qquad
      +\int_{-\infty}^s 
      \left(
      h_1(X(\eta),X'(\eta)) +h_2(X(\eta),X'(\eta),\eta) 
      \right)
      Z''(s-\eta) \, d\eta
      \bigg|
      \\
      &\le \frac{1}{2} \left( \frac{\alpha}{\beta} +1 \right) m(s)
      \le \frac{1}{2} \left( \frac{\alpha}{\beta} +1 \right) \delta_0=:b.
    \end{aligned}
  \end{equation}
  Combining \eqref{eq:3.18}--\eqref{eq:3.20},
  we infer that $\Psi$ is a mapping from $\mathcal{X}_{a,m}$ to $\mathcal{Y}_{a,b,m}$.
  From Lemma~\ref{lemma:3.1}, 
  we conclude that $\Psi$ is a continuous mapping 
  on the convex closed set $\mathcal{X}_{a,m}$ of $BC^1((-\infty,\tilde a])$
  the image of which is contained in the compact subset $\mathcal{Y}_{a,b,m}$ of $\mathcal{X}_{a,m}$.
  Therefore, by the Schauder fixed point theore,
  we obtain a fixed point $X$ of $\Psi$ in $\mathcal{X}_{a,m}$.
  The fixed point $X$ satisfies \eqref{eq:3.10}, 
  because \eqref{eq:3.16} and the fact that $X \in \mathcal{X}_{a,m}$ yield 
  $|X(s)| \le m(s) \to 0$ as $s \to -\infty$.
  Furthermore, since $X=\Psi(X)\in \mathcal{Y}_{a,b,m}$, 
  we also have $X\in C^2((-\infty,a])$.
  We have thus obtained the desired singular solution of \eqref{eq:3.1},
  and the proof is complete.
\end{proof}

\section{Known facts}
\label{section:4}

In this section, we collect some known facts to be used later.
More precisely, we recall properties of zeros of solutions for linear parabolic equations, 
gradient and H\"{o}lder estimates for semilinear parabolic equations,
and the intersection number of steady states of \eqref{eq:1.1}.


\subsection{Properties of the number of zeros}

We first introduce a known result providing some properties of the number of zeros
for radially symmetric solutions of linear parabolic equations.
To state the result precisely,
we define 
\begin{equation}
  \notag 
  Z_I[U] := \# \left\{
    r\in I:\, U(r) = 0
  \right\} \in \mathbb{N} \cup \{0\} \cup \{\infty\}
\end{equation}
for an interval $I\subset \mathbb{R}$ 
and a continuous function $U$ on $I$.
The following proposition is shown in \cite{CP}. 

\begin{proposition}[\cite{CP}]
  \label{proposition:4.1}
  Let $u(x,t) = U(|x|,t)$ be a non-trivial radially symmetric solution of the equation
  \begin{equation}
    \notag 
    \partial_t u = \Delta u + a(|x|,t)u, 
    \qquad 
    x\in B_R, \,\,\, t\in (t_1,t_2), 
  \end{equation}
  where $R>0$, $-\infty\le t_1<t_2\le \infty$ 
  and $a(r,t)$ is continuous function on $[0,R]\times (t_1,t_2)$. 
  Assume that
  \begin{equation*}
    U(R,t)\neq 0 \,\,\, \mbox{for all} \,\,\, t\in (t_1,t_2).
  \end{equation*}
  Then $Z_{[0,R)}[U(\cdot,t)]$ is finite for any $t\in (t_1,t_2)$. 
\end{proposition}

\subsection{Gradient estimate}


Next we recall a gradient estimate for a solution of semilinear heat equations 
obtained in \cite[Theorem~3.1]{FM}. 

\begin{proposition}[\cite{FM}]
  \label{proposition:4.2}
  Let $h$ be a non-negative $C^1$ function on $[0,\infty)$,
  $\Omega$ a ball in $\mathbb{R}^N$
  and $t_1$, $t_2$ real numbers with $t_1<t_2$. 
  Suppose that a function $u=u(x,t)$ 
  is non-negative and smooth on $\overline{\Omega} \times [t_1,t_2]$ 
  and satisfies  
\begin{equation*}
  \left\{
    \begin{aligned}
      &\partial_t u = \Delta u + h(u), 
      && x\in\Omega, \,\,\, t \in (t_1,t_2), 
      \\
      &u(x,t) = 0, 
      && x\in \partial\Omega, \,\,\, t \in (t_1,t_2),
      \\
      &\frac{1}{2} |\nabla u(x,t_1)|^2 \le \int_{u(x,t_1)}^L h(\eta)\, d\eta, 
      && x\in \Omega,
    \end{aligned}
  \right.
\end{equation*}
  where 
\begin{equation*}
L:= \max_{(x,t)\in\overline{\Omega} \times [t_1,t_2]}u(x,t).
\end{equation*}
  Then
  \begin{equation*}
    \frac{1}{2} |\nabla u(x,t)|^2 \le \int_{u(x,t)}^L h(\eta)\, d\eta
    \quad\mbox{for}\quad
    (x,t) \in \overline{\Omega} \times [t_1,t_2].
  \end{equation*}
\end{proposition}

\subsection{The intersection number of radially symmetric steady states}

Finally we introduce results for the intersection number of radially symmetric steady states
in the specific cases $f(u)=u^p$ and $f(u)=e^u$. 

First we consider the case $f(u)=u^p$.
In this case, 
radially symmetric steady states of \eqref{eq:1.1} are given by solutions of the equation
\begin{equation}
  \label{eq:4.1}
  \Phi''+\frac{N-1}{r}\Phi'+\Phi^p = 0,
  \qquad 
  r>0.
\end{equation}
As mentioned in the previous section,
this equation has the explicit singular solution $\Phi^*_p$ given by \eqref{eq:3.2},
provided $p>N/(N-2)$.
It was shown in \cite[Proposition~3.7]{Wang} (see also \cite{JL,Miyamoto}) that if $p$ is in some range,
$\Phi^*_p$ and any positive regular solution of \eqref{eq:4.1} intersect infinitely many times.
\begin{proposition}[\cite{Wang}]
  \label{proposition:4.3}
  Assume $N\ge 3$ and $p_{\mathrm{S}}<p<p_{\mathrm{JL}}$. 
  Let $\Phi$ be a solution of \eqref{eq:4.1} satisfying 
  $\Phi(0)=\alpha$ and $\Phi'(0)=0$ for some $\alpha>0$. 
  Then $\Phi$ is positive on $(0,\infty)$ and $Z_{(0,\infty)}[\Phi-\Phi_p^*]=\infty$. 
\end{proposition}

Let us move on to the case $f(u)=e^u$;
we consider the equation
\begin{equation}
  \label{eq:4.2}
  \Phi''+\frac{N-1}{r}\Phi'+e^\Phi = 0, 
  \qquad 
  r>0.
\end{equation}
The intersection number of solutions of \eqref{eq:4.2} was studied in \cite{JL,Tello,Miyamoto}.
We will use the following result stated in \cite[Theorem~1.1]{Tello}.

\begin{proposition}[\cite{Tello}]
  \label{proposition:4.4}
  Assume $3\le N\le 9$,
  and let $\Phi$ be a solution of \eqref{eq:4.2} satisfying
  $\Phi(0)=\alpha$ and $\Phi'(0)=0$ for some $\alpha \in \mathbb{R}$. 
  Then it holds that $Z_{(0,\infty)}[\Phi-\Phi_\infty^*]=\infty$. 
\end{proposition}

We remark that in Propositions~\ref{proposition:4.3} and \ref{proposition:4.4},
every zero of $\Phi-\Phi_p^*$ or $\Phi-\Phi_\infty^*$ is simple.

\section{Proof of Proposition~\ref{proposition:2.1}}
\label{section:5}

This section is devoted to the proof of Proposition~\ref{proposition:2.1}. 
Throughout the section,
we always assume (A1),
write $u(x,t)=U(|x|,t)$ for a radially symmetric solution of \eqref{eq:1.1}
and use $T \in (0,\infty]$ to denote its maximal existence time.

\subsection{Some lemmas}

Let us begin with collecting some lemmas to be used in the proof of Proposition~\ref{proposition:2.1}.

\begin{lemma}
  \label{lemma:5.1}
  There are $N_0 \in \mathbb{N} \cup \{0\}$ and $T_* \in (0,T)$ such that
  \begin{equation}
    \label{eq:5.1}
    \partial_t u(0,t) \neq 0,
    \qquad
    Z_{(0,R^*)}[U(\cdot,t)-U^*]=N_0
    \quad\mbox{for}\quad
    t \in [T_*,T),
  \end{equation}
  where $U^* \in C^2(0,r_0)$ is the singular solution 
  obtained in Proposition~\textup{\ref{proposition:3.1}} and $R^*:=\min\{r_0,R\}$.
\end{lemma}

\begin{lemma}
  \label{lemma:5.2}
  Let $w=w(y,\tau)$ be a smooth radially symmetric function satisfying
  \begin{gather}
    \label{eq:5.2}
    \partial_\tau w = \Delta_y w + h(w)
    \quad \mbox{in} \quad 
    \mathbb{R}^N \times I, 
    \\
    \label{eq:5.3}
    \partial_\tau w(0,\tau)=0
    \quad \mbox{for} \quad 
    \tau \in I,
  \end{gather}
  where $I \subset \mathbb{R}$ is an open interval and $h$ is a $C^1$ function.
  Then 
  \begin{equation*}
    \partial_\tau w=0
    \quad \mbox{in} \quad 
    \mathbb{R}^N \times I.
  \end{equation*}
\end{lemma}

Lemma~\ref{lemma:5.1} is shown in \cite[Lemma~2.10, Corollary~2.11]{MM} 
in the case $f(u)=u^p$,
and the same proof is valid also for general $f(u)$.
Therefore we only give the proof of Lemma~\ref{lemma:5.2}. 

\begin{proof}[Proof of Lemma~\ref{lemma:5.2}]
  Write $w(y,\tau)=W(|y|,\tau)$ 
  and set $\tilde w:=\partial_\tau w$, $\tilde W:=\partial_\tau W$. 
  Then by \eqref{eq:5.2} we see that $\tilde w$ satisfies the linear parabolic equation
  \begin{equation*}
    \partial_\tau \tilde w = \Delta_y \tilde w + h'(W(|y|,\tau)) \tilde w
    \quad \mbox{in} \quad 
    \mathbb{R}^N \times I.
  \end{equation*} 

  Contrary to the claim,
  suppose that $\tilde W(\delta_0,\tau_0) \neq 0$ 
  for some $(\delta_0,\tau_0) \in [0,\infty) \times I$.
  Then the assumption \eqref{eq:5.3} gives $\delta_0>0$.
  By continuity, 
  we can take $\varepsilon_0>0$ such that 
  \begin{equation*}
    \tilde W (\delta_0,\tau) \neq 0
    \quad \mbox{for} \quad 
    \tau\in [\tau_0-\varepsilon_0,\tau_0+\varepsilon_0].
  \end{equation*} 
  Applying Proposition~\ref{proposition:4.1},
  we deduce that $Z_{[0,\delta_0]}[\tilde W(\cdot,\tau_0)]<\infty$.
  Therefore there is $\delta_1>0$ 
  such that $\tilde W(\rho,\tau_0) \neq 0$ for $\rho \in (0,\delta_1]$.
  By replacing $\tilde W$ with $-\tilde W$ if necessary,
  we may assume that
  \begin{equation}
    \label{eq:5.4}
    \tilde W(\rho,\tau_0)>0
    \quad \mbox{for} \quad 
    \rho \in (0,\delta_1].
  \end{equation}
  This particularly gives $\tilde W(\delta_1,\tau_0)>0$,
  and hence by continuity there is $\varepsilon_1>0$ such that
  \begin{equation}
    \label{eq:5.5}
    \tilde W(\delta_1,\tau)>0
    \quad \mbox{for} \quad 
    \tau \in [\tau_0,\tau_0+\varepsilon_1]. 
  \end{equation}
  From \eqref{eq:5.4} and \eqref{eq:5.5},
  we can apply the strong maximum principle to obtain
  \begin{equation}
    \notag 
    \tilde W(\rho,\tau)>0
    \quad \mbox{for} \quad 
    (\rho,\tau) \in [0,\delta_1] \times (\tau_0,\tau_0+\varepsilon_1].
  \end{equation}
  This contradicts the assumption \eqref{eq:5.3},
  and the lemma follows.
\end{proof}

\subsection{Bounds for rescaled solutions}

For $t \in (0,T)$, 
we consider the rescaled functions
\begin{equation*}
  v^{t}(y,\tau):=\frac{F\big( u(\lambda (t)y,t+\lambda (t)^2 \tau) \big)}{\lambda (t)^2}
  \quad \mbox{for} \quad 
  (y,\tau) \in B_{R/\lambda(t)} \times I^{t},
\end{equation*}
where
\begin{equation*}
  \lambda (t):=F(u(0,t))^{\frac{1}{2}},
  \qquad
  I^t:=\left( -\frac{t}{\lambda(t)^2}, \frac{T-t}{\lambda(t)^2}\right).
\end{equation*}
The aim of this subsection is to derive upper and lower bounds for $v^{t}$.

\begin{proposition}
  \label{proposition:5.1}
  Assume \textup{(A2)}, \eqref{eq:1.8} and
  \begin{equation}
    \label{eq:5.6}
    \lim_{t \to T} u(0,t)=\infty.
  \end{equation}
  Then for any $\rho_0>0$ and $\tau_0 \in (0,1)$,
  there are constants $c>0$, $C>0$ and $T^* \in [T_0,T)$ such that
  \begin{gather}
    \label{eq:5.7}
    \overline{B_{\rho_0}} \times [-\tau_0,\tau_0] \subset B_{R/\lambda(t)} \times I^{t}
    \quad \mbox{for} \quad t \in [T^*,T),
    \\
    \label{eq:5.8}
    c \le v^{t}(y,\tau) \le C,
    \quad
    |\nabla v^{t}(y,\tau)| \le C
    \quad \mbox{for} \quad 
    (y,\tau) \in \overline{B_{\rho_0}} \times [-\tau_0,\tau_0], \,\,\, t \in [T^*,T).
  \end{gather}
\end{proposition}

We first consider estimates for $\lambda(t)$.
\begin{lemma}
  \label{lemma:5.3}
  Assume \eqref{eq:1.8}.
  Then for $t \in [T_0,T)$ and $\tau \in I^t$ with $t+\lambda(t)^2\tau \ge T_0$,
  it holds that
  \begin{equation}
    \label{eq:5.9}
    1-|\tau| \le \frac{\lambda(t+\lambda(t)^2\tau)^2}{\lambda(t)^2} \le 1+|\tau|,
    \qquad
    \frac{T-t}{\lambda(t)^2} \ge 1.
  \end{equation}
\end{lemma}

\begin{proof}
  Since the assumption \eqref{eq:1.8} means that 
  $x=0$ is a maximum point of $u(\cdot,s)$, 
  we have $\Delta u(0,s)\le 0$ for $s \in [T_0,T)$.
  Hence
  \begin{equation*}
    \partial_t u(0,s)=\Delta u(0,s)+f(u(0,s)) \le f(u(0,s))
    \quad \mbox{for} \quad s \in [T_0,T).
  \end{equation*}
  From this and \eqref{eq:5.1},
  we infer that
  \begin{equation}
    \label{eq:5.10}
    0<
    -\frac{d}{ds} \lambda (s)^2
    =\frac{\partial_t u(0,s)}{f(u(0,s))} 
    \le 1
    \quad \mbox{for} \quad s \in [T_0,T).
  \end{equation}
  This gives the former of \eqref{eq:5.9}, since
  \begin{equation*}
    \left|
    \frac{\lambda(t+\lambda(t)\tau)^2}{\lambda(t)^2}-1
    \right|
    = \frac{|\lambda(t+\lambda(t)^2 \tau)^2 -\lambda(t)^2|}{\lambda(t)^2}
    =\frac{1}{\lambda(t)^2}
    \left|
    \int_{t}^{t+\lambda(t)^2 \tau} \frac{d}{ds} \lambda (s)^2 ds
    \right|
    \le |\tau|.
  \end{equation*}
  
  Let us prove the latter of \eqref{eq:5.9}.
  Since the desired estimate is obviously true for $T=\infty$,
  we only need to consider the case $T<\infty$.
  Then $u$ must blow up at $T$; 
  there is $\{t_i\} \subset [T_0,T)$ 
  such that $t_i \to T$ and $\| u(\cdot,t_i)\|_{L^\infty(\Omega)} =u(0,t_i) \to \infty$ as $i \to \infty$.
  In particular, we have $\lambda (t_i)=F(u(0,t_i))^{1/2} \to 0$ as $i \to \infty$.
  Integrating \eqref{eq:5.10} from $t$ to $t_i$ and letting $i \to \infty$,
  we deduce that $\lambda (t)^2 \le T-t$.
  We have thus verified the latter of \eqref{eq:5.9},
  and the lemma follows.
\end{proof}

\begin{lemma}
  \label{lemma:5.4}
  Assume \eqref{eq:1.8} and \eqref{eq:5.6}.
  Then there exists $T_1 \in [T_0,T)$ such that
  \begin{equation}
    \label{eq:5.11}
    \frac{1}{2} |\nabla u(x,t)|^2 \le \int_{u(x,t)}^{u(0,t)} f(\eta)\, d\eta 
    \quad\mbox{for}\quad (x,t) \in \Omega \times [T_1,T).
  \end{equation}
\end{lemma}

\begin{proof}
  By Lemma~\ref{lemma:5.1} and the assumption \eqref{eq:5.6},
  we have $\partial_t u(0,t)>0$ for $t \in [T_*,T)$,
  where $T_*$ is the number given in Lemma~\ref{lemma:5.1}.
  This with the assumption \eqref{eq:1.8} shows that
  \begin{equation}
    \label{eq:5.12}
    \max_{(x,t) \in \overline{\Omega} \times [T_*,t_0]} u(x,t) =u(0,t_0)
    \quad\mbox{for}\quad t_0 \in [T_*,T).
  \end{equation} 
  One easily sees that \eqref{eq:1.2} implies
  \begin{equation*}
    \int_0^\infty f(\eta) \, d\eta=\infty.
  \end{equation*}
  Hence by the assumption \eqref{eq:5.6},
  we can take $T_1 \in [T_*,T)$ such that
  \begin{equation*}
    \sup_{x \in \Omega} 
    \left(
    \frac{1}{2} |\nabla u(x,T_*)|^2 +\int_0^{u(x,T_*)} f(\eta)\, d\eta
    \right)
    \le
    \int_0^{u(0,t_0)} f(\eta)\, d\eta
    \quad\mbox{for}\quad t_0 \in [T_1,T),
  \end{equation*}
  which leads to
  \begin{equation*}
    \frac{1}{2} |\nabla u(x,T_*)|^2
    \le \int_{u(x,T_*)}^{u(0,t_0)} f(\eta)\, d\eta
    \quad\mbox{for}\quad x \in \Omega, \ t_0 \in [T_1,T).
  \end{equation*}
  From this and \eqref{eq:5.12},
  we see that Proposition~\ref{proposition:4.2} can be applied for $t_1=T_*$.
  We then conclude that
  \begin{equation*}
    \frac{1}{2} |\nabla u(x,t)|^2
    \le \int_{u(x,t)}^{u(0,t_0)} f(\eta)\, d\eta
    \quad\mbox{for}\quad (x,t) \in \Omega \times [T_*,t_0],
  \end{equation*}
  provided $t_0 \in [T_1,T)$.
  Choosing $t=t_0$ thus gives the desired inequality.
\end{proof}

Using the previous lemma,
we derive an estimate for the ratio $u(x,t)/u(0,t)$.
\begin{lemma}
  \label{lemma:5.5}
  Suppose that \textup{(A2)}, \eqref{eq:1.8} and \eqref{eq:5.6} hold.
  Then
  \begin{equation}
    \label{eq:5.13}
    \liminf_{t \to T} \inf_{|y| \le \rho_0, |\tau| \le \tau_0} 
    \frac{u(\lambda(t)y,t+\lambda(t)^2\tau)}{u(0,t+\lambda(t)^2\tau)}
    \ge 1 -\frac{2(q-1)\rho_0^2}{1-\tau_0}
  \end{equation}
  for $\rho_0>0$ and $\tau_0 \in (0,1)$.
\end{lemma}

\begin{proof}
  By (A2),
  we can take $f_0 \in C^1([0,\infty))$ 
  satisfying \eqref{eq:1.7}.
  Then we may assume that
  \begin{equation}
    \label{eq:5.14}
    \frac{f_0(\eta)}{2} \le f(\eta) \le 2f_0(\eta),
    \qquad
    \frac{F_0(\eta)}{2} \le F(\eta) \le 2F_0(\eta)
    \quad \mbox{for} \quad \eta>0.
  \end{equation}
  Indeed, if a positive number $\eta_0$ is taken so large that
  \begin{equation*}
    \frac{1}{2} \le \frac{f_0(\eta)}{f(\eta)} \le 2
    \quad \mbox{for} \quad \eta \ge \eta_0,
  \end{equation*}  
  and a cut-off function $\zeta \in C^\infty(\mathbb{R})$ is chosen such that
  \begin{equation*}
    0 \le \zeta \le 1,
    \qquad
    \zeta(\eta)=0
    \quad \mbox{for} \quad \eta \le \eta_0,
    \qquad
    \zeta(\eta)=1
    \quad \mbox{for} \quad \eta \ge \eta_0+1,
  \end{equation*}  
  then we can easily check that \eqref{eq:5.14}, 
  as well as \eqref{eq:1.7},
  is satisfied if $f_0$ is replaced by $\zeta f_0 +(1-\zeta) f$.

  Note that \eqref{eq:1.7} implies
  \begin{equation}
    \label{eq:5.15}
    \lim_{\eta \to \infty} f_0'(\eta)=\infty.
  \end{equation}
  This particularly shows that 
  $f_0$ is non-decreasing on $[\eta_1,\infty)$ for some $\eta_1 \in (0,\infty)$, 
  namely
  \begin{equation*}
    f_0(\eta) \le f_0(\tilde \eta)
    \quad \mbox{if} \quad 
    \tilde \eta \ge \eta \ge \eta_1.
  \end{equation*}
  Moreover, since \eqref{eq:5.15} gives $f_0(\eta) \to \infty$ as $\eta \to \infty$,
  there is $\eta_2 \in [\eta_1,\infty)$ such that  
  \begin{equation*}
    f_0(\tilde \eta) \ge \max_{\eta \in [0,\eta_1]} f_0(\eta)
    \quad \mbox{if} \quad 
    \tilde \eta \ge \eta_2.
  \end{equation*}
  Combining the above,
  we deduce that
  \begin{equation}
    \label{eq:5.16}
    f_0(\eta) \le f_0(\tilde \eta)
    \quad \mbox{if} \quad 
    \tilde \eta \ge \eta_2, \,\,\, \tilde \eta \ge \eta.
  \end{equation}
  We take $\tilde T_0 \in [T_0,T)$ such that $u(0,s) \ge \eta_2$ for all $s \in [\tilde T_0,T)$.
  Such $\tilde T_0$ indeed exists by the assumption \eqref{eq:5.6}.
  Using \eqref{eq:5.14} and then applying \eqref{eq:5.16} for $\tilde \eta=u(0,s)$,
  we have
  \begin{equation*}
    \int_{u(z,s)}^{u(0,s)} f(\eta)\, d\eta 
    \le
    2\int_{u(z,s)}^{u(0,s)} f_0(\eta)\, d\eta 
    \le
    2f_0(u(0,s))(u(0,s)-u(z,s))
  \end{equation*}
  for all $(z,s) \in B_R \times [\tilde T_0,T)$.
  Hence we see from Lemma~\ref{lemma:5.4} that
  \begin{equation*}
    \frac{1}{2}|\nabla u(z,s)| 
    \le
    f_0(u(0,s))^{\frac{1}{2}}(u(0,s)-u(z,s))^{\frac{1}{2}} 
  \end{equation*}
  if $s \in [{\max \{ \tilde T_0,T_1\}},T)$.
  We put $z=\xi x/|x|$ ($\xi \in (0,R)$, $x \in B_R$) and plug the inequality
  \begin{equation}
    \label{eq:5.17}
    -\frac{d}{d\xi} u
    \left(
    \xi \frac{x}{|x|},s
    \right)
    \le
    \left|
    \nabla u
    \left(
    \xi \frac{x}{|x|},s
    \right)
    \right|
  \end{equation}
  into the left-hand side of the above inequality.
  Then we have
  \begin{equation*}
    \frac{d}{d\xi}
    \left(
    u(0,s)-u
    \left(
    \xi \frac{x}{|x|},s
    \right)
    \right)^{\frac{1}{2}}
    \le
    f_0(u(0,s))^{\frac{1}{2}}.
  \end{equation*}
  Integrating this with respect to $\xi$ over $[0,|x|]$ and squaring the result give
  \begin{equation}
    \label{eq:5.18}
    u(0,s)-u(x,s)
    \le
    f_0(u(0,s)) |x|^2.
  \end{equation}
  Applying this with $(x,s)=(\lambda(t)y,t+\lambda(t)^2\tau)$,
  we obtain
  \begin{align}
    \label{eq:5.19}
    \begin{aligned}
      \left.
      \frac{u(x,s)}{u(0,s)}
      \right|_{(x,s)=(\lambda(t)y,t+\lambda(t)^2\tau)}
      &\ge 
      \left.
      1-\frac{f_0(u(0,s)|x|^2}{u(0,s)}
      \right|_{(x,s)=(\lambda(t)y,t+\lambda(t)^2\tau)}
      \\
      &=1 -
      \left.
      \frac{f_0(u(0,s)) F(u(0,s))}{u(0,s)}
      \right|_{s=t+\lambda(t)^2\tau}
      \cdot
      \frac{\lambda(t)^2 |y|^2}{\lambda(t+\lambda(t)^2\tau)^2}
      \\
      &\ge 1 -
      \left.
      \frac{2f_0(u(0,s)) F_0(u(0,s))}{u(0,s)}
      \right|_{s=t+\lambda(t)^2\tau}
      \cdot
      \frac{|y|^2}{1-|\tau|},
    \end{aligned}
  \end{align}
  where we have used \eqref{eq:5.9} and \eqref{eq:5.14} in the last inequality.
  Since L'Hospital's rule yields
  \begin{equation}
    \label{eq:5.20}
    \lim_{\eta \to \infty} \frac{f_0(\eta) F_0(\eta)}{\eta}
    =\lim_{\eta \to \infty} \frac{d}{d\eta} f_0(\eta) F_0(\eta)
    =\lim_{\eta \to \infty} (f_0'(\eta) F_0(\eta)-1)
    =q-1,
  \end{equation}
  we see that the right-hand side of \eqref{eq:5.19} converges to $1-2(q-1)|y|^2/(1-|\tau|)$
  locally uniformly for $(y,\tau) \in \mathbb{R}^N \times (-1,1)$ as $t \to T$.
  Thus \eqref{eq:5.13} holds, as claimed.
\end{proof}

We need an additional lemma to prove Proposition~\ref{proposition:5.1} in the case $q>1$.
\begin{lemma}
  \label{lemma:5.6}
  Suppose that \textup{(A2)} holds for some $q>1$.
  Then there exist positive constants 
  $C$, $\alpha_1$ and $\alpha_2$ with $\alpha_2>\alpha_1$ such that
  \begin{equation*}
    \frac{\eta}{C} \le f(\eta)F(\eta) \le C\eta,
    \qquad
    \left(
    \frac{\eta_2}{\eta_1}
    \right)^{\alpha_1}
    \le 
    \frac{F(\eta_1)}{F(\eta_2)}
    \le 
    \left(
    \frac{\eta_2}{\eta_1}
    \right)^{\alpha_2}
    \quad \mbox{if} \quad \eta_2 \ge \eta_1 \ge 1.
  \end{equation*}
\end{lemma}

\begin{proof}
  We take $f_0$ satisfying \eqref{eq:1.7} and \eqref{eq:5.14}.
  Then the assumption $q>1$ shows that the limit computed in \eqref{eq:5.20} is positive.
  Hence we can take positive constants $c$ and $C$ such that
  \begin{equation*}
    \frac{c}{\eta} \le \frac{1}{f_0(\eta)F_0(\eta)} \le \frac{C}{\eta}
    \quad \mbox{for} \quad
    \eta \ge 1.
  \end{equation*}
  From \eqref{eq:5.14}, 
  it follows that
  \begin{equation}
    \label{eq:5.21}
    \frac{c}{4\eta} \le \frac{1}{f(\eta)F(\eta)} \le \frac{4C}{\eta}
    \quad \mbox{for} \quad
    \eta \ge 1.
  \end{equation}  
  Plugging the equality
  \begin{equation*}
    \frac{1}{f(\eta)F(\eta)}=-\frac{d}{d\eta} \log F(\eta)
  \end{equation*}
  into \eqref{eq:5.21} and integrating the result over $[\eta_1,\eta_2]$,
  we deduce that
  \begin{equation*}
    \frac{c}{4} \log \frac{\eta_2}{\eta_1} 
    \le \log \frac{F(\eta_1)}{F(\eta_2)} 
    \le 4C \log \frac{\eta_2}{\eta_1}.
  \end{equation*}
  Therefore the lemma follows.
\end{proof}

We are now in a position to prove Proposition~\ref{proposition:5.1}. 
We discuss the cases $q=1$ and $q>1$ separately.

\begin{proof}[Proof of Proposition~\ref{proposition:5.1} for $q=1$]
  Let $\rho_0>0$ and $\tau_0 \in (0,1)$.
  Since the assumption \eqref{eq:5.6} gives
  $\lambda(t)=F(u(0,t))^{1/2} \to 0$ as $t \to T$,
  we can choose $T_2 \in [T_0,T)$ such that 
  $R/\lambda(t) \ge \rho_0$, $t/\lambda(t) \ge \tau_0$ for $t \in [T_2,T)$.
  This with \eqref{eq:5.9} shows that \eqref{eq:5.7} holds for $t \in [T_2,T)$.  

  Let us derive the lower estimate of $v^{t}$ in \eqref{eq:5.8}.
  From the assumption \eqref{eq:1.8}
  and the fact that $F$ is monotonically decreasing,
  we have
  \begin{equation*}
    \frac{F(u(x,s))}{F(u(0,s))} \ge 1
    \quad \mbox{for} \quad (x,s) \in B_{R} \times [T_0,T).
  \end{equation*}
  By definition,
  $v^{t}(y,\tau)$ is written as 
  \begin{align}
    \notag
    v^{t}(y,\tau) 
    &=\frac{\lambda(t+\lambda(t)^2 \tau)^2}{\lambda(t)^2}
    \cdot 
    \frac{F\big( u(\lambda (t)y,t+\lambda (t)^2 \tau) \big)}{\lambda(t+\lambda(t)^2 \tau)^2}
    \\
    \label{eq:5.22}
    &=\frac{\lambda(t+\lambda(t)^2 \tau)^2}{\lambda(t)^2}
    \cdot \left.
    \frac{F(u(x,s))}{F(u(0,s))}
    \right|_{(x,s)=(\lambda(t)y,t+\lambda(t)^2 \tau)}.
  \end{align}
  Combining the above with \eqref{eq:5.9},
  we find that
  \begin{equation*}
    v^{t}(y,\tau) \ge 1-\tau_0
    \quad \mbox{for} \quad (y,\tau) \in B_{R/\lambda(t)} \times [-\tau_0,\tau_0].
  \end{equation*}
  This is valid under the conditions $t \ge T_0$ and $t+\lambda(t)^2 \tau \ge T_0$,
  which are satisfied if $t$ is close to $T$.
  Thus we obtain the desired lower estimate.

  Next we consider the upper estimate of $v^{t}$.
  Take $f_0$ satisfying \eqref{eq:1.7} and \eqref{eq:5.14},
  and put
  \begin{equation*}
    m(t):=\inf_{|y| \le \rho_0, |\tau| \le \tau_0} 
    u(\lambda(t)y,t+\lambda(t)\tau),
    \qquad
    \kappa(t):=\sup_{\eta \in [m(t),\infty)} f_0'(\eta)F_0(\eta). 
  \end{equation*}
  Then we see from \eqref{eq:5.6}, \eqref{eq:5.13} 
  and the assumption $q=1$ that
  \begin{equation*}
    \liminf_{t \to \infty} m(t)
    =\liminf_{t \to \infty} \inf_{|y| \le \rho_0, |\tau| \le \tau_0} 
    u(0,t+\lambda(t)^2\tau) \cdot 
    \frac{u(\lambda(t)y,t+\lambda(t)^2\tau)}{u(0,t+\lambda(t)^2\tau)}
    =\infty.
  \end{equation*}
  Hence
  \begin{equation}
    \label{eq:5.23}
    \lim_{t \to T} \kappa (t)=\lim_{\eta \to \infty} f_0'(\eta)F_0(\eta)=1.
  \end{equation}
  By definition, we have
  \begin{equation}
    \label{eq:5.24}
    f_0'(\eta)F_0(\eta) \le \kappa(t)
    \quad \mbox{if} \quad \eta \ge m(t).
  \end{equation}
  This means that 
  the function $\eta \mapsto f_0(\eta)F_0(\eta)^{\kappa(t)}$ 
  is non-increasing on $[m(t),\infty)$,
  since
  \begin{equation*}
    \frac{d}{d\eta}(f_0(\eta)F_0(\eta)^{\kappa(t)})
    =F_0(\eta)^{\kappa(t)-1}(f_0'(\eta)F_0(\eta)-\kappa(t)).
  \end{equation*}
  From this and \eqref{eq:5.14},
  we find that
  \begin{equation*}
    \frac{f(\eta_2)F(\eta_2)^{\kappa(t)}}{f(\eta_1)F(\eta_1)^{\kappa(t)}}
    \le 4^{1+\kappa(t)} \cdot 
    \frac{f_0(\eta_2)F_0(\eta_2)^{\kappa(t)}}{f_0(\eta_1)F_0(\eta_1)^{\kappa(t)}}
    \le 4^{1+\kappa(t)}
    \quad \mbox{if} \quad \eta_2 \ge \eta_1 \ge m(t).
  \end{equation*}
  Therefore, 
  for $a,b \in \mathbb{R}$ with $b \ge a \ge m(t)$,
  \begin{align*}
    \int_a^b f(\eta)\, d\eta 
    &=
    f(a)^2 F(a)^{2\kappa(t)} 
    \int_{a}^{b} 
    \left( 
    \frac{f(\eta)F(\eta)^{\kappa(t)}}{f(a)F(a)^{\kappa(t)}}
    \right)^2
    \cdot \frac{1}{f(\eta)F(\eta)^{2\kappa(t)}}
    \, d\eta 
    \\
    &\le 
    4^{2(1+\kappa(t))} f(a)^2 F(a)^{2\kappa(t)} 
    \int_{a}^{b}
    \frac{1}{f(\eta)F(\eta)^{2\kappa(t)}}
    \, d\eta 
    \\
    &=
    \frac{4^{2(1+\kappa(t))} f(a)^2 F(a)^{2\kappa(t)}}{2\kappa(t)-1}
    \left(
    \frac{1}{F(b)^{2\kappa(t)-1}} -\frac{1}{F(a)^{2\kappa(t)-1}}
    \right),
  \end{align*}
  where we have used the fact that 
  $(2\kappa(t) -1)/(f_0(\eta) F_0(\eta)^{2\kappa(t)})=(1/F_0(\eta)^{2\kappa(t)-1})'$
  in deriving the last equality.
  Setting $a=u(z,s)$ and $b=u(0,s)$ in the above,
  and then plugging the result into \eqref{eq:5.11},
  we obtain
  \begin{equation}
    \label{eq:5.25}
    \frac{(\kappa(t)-1/2)^{1/2}|\nabla u(z,s)|}{4^{1+\kappa(t)} f(u(z,s))F(u(z,s))^{\kappa(t)}} 
    \left(
    \frac{1}{\lambda (s)^{2(2\kappa(t)-1)}} -\frac{1}{F(u(z,s))^{2\kappa(t)-1}}
    \right)^{-1/2}
    \le 1,
  \end{equation}
  provided $u(z,s) \ge m(t)$. 
  As in the derivation of \eqref{eq:5.18},
  we put $z=\xi x/|x|$,
  then apply \eqref{eq:5.17},
  and finally integrate the resulting inequality with respect to $\xi$ over $[0,|x|]$.
  The result is
  \begin{equation}
    \label{eq:5.26}
    \int_{u(x,s)}^{u(0,s)} \frac{4^{1+\kappa(t)} (\kappa(t)-1/2)^{1/2}}{f(\eta)F(\eta)^{\kappa(t)}} 
    \left(
    \frac{1}{\lambda(s)^{2(2\kappa(t)-1)}} -\frac{1}{F(\eta)^{2\kappa(t)-1}}
    \right)^{-1/2}
    \, d\eta
    \le |x|.
  \end{equation}
  Using the change of variables $\zeta=\lambda(s)^{-2}F(\eta)$,
  we arrive at
  \begin{equation}
    \label{eq:5.27}
    \psi_{\kappa(t)}
    \left(
    \frac{F(u(x,s))}{\lambda(s)^2}
    \right)
    \le \frac{|x|}{\lambda(s)}
    \quad\mbox{if}\quad 
    u(z,s) \ge m(t),
  \end{equation}
  where
  \begin{equation}
    \label{eq:5.28}
    \psi_\gamma (\xi):=\int_1^\xi
    \frac{4^{1+\gamma}(\gamma-1/2)^{1/2}}{\{ \zeta(\zeta^{2\gamma-1} -1)\}^{1/2}}
    \, d\zeta
    \quad \mbox{for} \quad 
    \gamma>\frac{1}{2}, \, \xi \ge 1.
  \end{equation}
  By the definition of $m(t)$, we see that the condition $u(x,s) \ge m(t)$ is satisfied
  if $(x,s)=(\lambda(t)y,t+\lambda(t)^2 \tau)$, $|y| \le \rho_0$, $|\tau| \le \tau_0$.
  Hence putting $(x,s)=(\lambda(t)y,t+\lambda(t)^2 \tau)$ in \eqref{eq:5.27} yields
  \begin{equation}
    \label{eq:5.29}
    \psi_{\kappa(t)}
    \left(
    \frac{F(u(\lambda(t)y,t+\lambda(t)^2 \tau))}{\lambda(t+\lambda(t)^2 \tau)^2}
    \right)
    \le \frac{\lambda(t)|y|}{\lambda(t+\lambda(t)^2 \tau)}
    \quad\mbox{if}\quad 
    |y| \le \rho_0, \,\,\, |\tau| \le \tau_0.
  \end{equation}
  Since $\psi_\gamma'(\xi)>0$ for $\xi>1$, 
  we see that $\psi_\gamma$ has an inverse with the domain $[0,\rho_\gamma)$,
  where
  \begin{equation*}
    \rho_\gamma :=\sup_{\xi \ge 1} \psi_\gamma (\xi) 
    =\int_1^\infty
    \frac{4^{1+\gamma}(\gamma-1/2)^{1/2}}{\{ \zeta(\zeta^{2\gamma-1} -1)\}^{1/2}}
    \, d\zeta.
  \end{equation*}
  By \eqref{eq:5.23} and Fatou's lemma, 
  we obtain 
  \begin{gather*}
    \liminf_{t \to T} \rho_{\kappa(t)}
    \ge 
    \int_1^\infty
    \frac{16 \cdot 2^{-1/2}}{\{ \zeta(\zeta -1)\}^{1/2}}
    \, d\zeta
    =\infty.
  \end{gather*}
  This shows that the domain of the inverse $\psi_{\kappa(t)}^{-1}$ 
  contains the interval $[0,\rho_0/(1-\tau_0)^{1/2}]$ if $t$ is close to $T$.
  Hence, from \eqref{eq:5.9} and \eqref{eq:5.29},
  we conclude that
  \begin{align*}
    \sup_{|y| \le \rho_0, |\tau| \le \tau_0}
    \frac{F(u(\lambda(t)y,t+\lambda(t)^2 \tau))}{\lambda(t+\lambda(t)^2 \tau)^2}
    &\le 
    \sup_{|y| \le \rho_0, |\tau| \le \tau_0} 
    \psi_{\kappa(t)}^{-1} 
    \left(
    \frac{\lambda(t)|y|}{\lambda(t+\lambda(t)^2 \tau)}
    \right)
    \\
    &\le
    \psi_{\kappa(t)}^{-1} 
    \left(
    \frac{\rho_0}{(1-\tau_0)^{1/2}}
    \right)
    \to
    \psi_1^{-1}
    \left(
    \frac{\rho_0}{(1-\tau_0)^{1/2}}
    \right)
    \quad\mbox{as}\quad t \to T.
  \end{align*}
  Combining this with \eqref{eq:5.9} and \eqref{eq:5.22},
  we obtain
  \begin{equation}
    \label{eq:5.30}
    \limsup_{t \to T} \sup_{|y| \le \rho_0, |\tau| \le \tau_0} 
    v^{t}(y,\tau)
    \le K_0:=(1+\tau_0) \psi_1^{-1}
    \left(
    \frac{\rho_0}{(1-\tau_0)^{1/2}}
    \right).
  \end{equation}
  We have thus shown the upper estimate of $v^{t}$.
  
  It remains to derive the bound of $|\nabla v^{t}|$.
  Note that
  \begin{equation*}
    |\nabla v^{t}(y,\tau)|
    =\frac{1}{\lambda(t)} \cdot
    \left.
    \frac{|\nabla_x u(x,s)|}{f(u(x,s))}
    \right|_{(x,s)=(\lambda (t)y,t+\lambda (t)^2 \tau)}.
  \end{equation*}
  Applying \eqref{eq:5.25} to the right-hand side of the above equality, 
  we have
  \begin{align*}
    |\nabla v^{t}(y,\tau)|
    &\le 
    \left.
    \frac{4^{1+\kappa(t)}F(u(x,s))^{\kappa(t)}}{(\kappa(t)-1/2)^{1/2}\lambda(t)}
    \left(
    \frac{1}{\lambda (s)^{2(2\kappa(t)-1)}} -\frac{1}{F(u(x,s))^{2\kappa(t)-1}}
    \right)^{1/2}
    \right|_{(x,s)=(\lambda (t)y,t+\lambda (t)^2 \tau)}
    \\
    &\le 
    \left.
    \frac{4^{1+\kappa(t)}F(u(x,s))^{\kappa(t)}}{(\kappa(t)-1/2)^{1/2}\lambda(t)}
    \cdot \frac{1}{\lambda (s)^{2\kappa(t)-1}}
    \right|_{(x,s)=(\lambda (t)y,t+\lambda (t)^2 \tau)}
    \\
    &=
    \frac{4^{1+\kappa(t)}}{(\kappa(t)-1/2)^{1/2}} \cdot 
    \left(
    \frac{\lambda(t)}{\lambda(t+\lambda (t)^2 \tau)}
    \right)^{2\kappa(t)-1}
    \cdot v^{t}(y,\tau)^{\kappa(t)}.
  \end{align*}
  This together with \eqref{eq:5.9} and \eqref{eq:5.30} gives
  \begin{equation}
    \label{eq:5.31}
    \limsup_{t \to T} \sup_{|y| \le \rho_0, |\tau| \le \tau_0} 
    |\nabla v^{t}(y,\tau)| 
    \le 
    \left.
    \frac{4^{1+\gamma}}{(\gamma-1/2)^{1/2}} \cdot 
    \left( 
    \frac{1}{1-\tau_0} 
    \right)^{\gamma-1/2} 
    \cdot K_0^{\gamma}
    \, \right|_{\gamma=1}.
  \end{equation}  
  Thus \eqref{eq:5.8} is verified,
  and the proof is complete.
\end{proof}

Let us proceed to the case $q>1$.

\begin{proof}[Proof of Proposition~\ref{proposition:5.1} for $q>1$]
  We omit the derivation of \eqref{eq:5.7} and the lower estimate of $v^{t}$,
  because it can be done in the same way as in the case $q=1$.

  Let us consider the upper estimate of $v^{t}$.
  Fix $\rho_0>0$ and $\tau_0 \in (0,1)$.
  %
  We introduce a rescaled function $w^{t}$ defined by 
  \begin{equation*}
    w^{t}(y,\tau):=\frac{u(\lambda (t)y,t+\lambda (t)^2 \tau)}{u(0,t)}
    \quad \mbox{for} \quad 
    (y,\tau) \in B_{R/\lambda(t)} \times I^t,
  \end{equation*}
  and check that $w^{t}$ satisfies
  \begin{gather}
    \label{eq:5.32}
    \partial_\tau w^t  \ge \Delta_y w^t
    \quad \mbox{in} \quad B_{R/\lambda(t)} \times I^t,
    \\
    \label{eq:5.33}
    \underline{w}_0 
    :=\liminf_{t \to T} \inf_{|y| \le \rho_*, |\tau|\le \tau_1} w^t(y,\tau)>0,
  \end{gather}
  where $\tau_1$ is a fixed number satisfying $\tau_0<\tau_1<1$
  and $\rho_*$ is defined by
  \begin{equation*}
    \rho_*:=\frac{(1-\tau_1)^{1/2}}{2(q-1)^{1/2}}.
  \end{equation*}
  The inequality \eqref{eq:5.32} is easily obtained by the fact that $\partial_t u -\Delta u =f(u) \ge 0$.
  Let us derive \eqref{eq:5.33}.
  We apply Lemma~\ref{lemma:5.5} to obtain
  \begin{equation}
    \label{eq:5.34}
    \liminf_{t \to T} \inf_{|y| \le \rho_*, |\tau| \le \tau_1} 
    \frac{u(\lambda(t)y,t+\lambda(t)^2\tau)}{u(0,t+\lambda(t)^2\tau)}
    \ge 1 -\frac{2(q-1)\rho_*^2}{1-\tau_1}=\frac{1}{2}.
  \end{equation}
  Since \eqref{eq:5.6} gives $u(0,t+\lambda(t)^2 \tau) \to \infty$ 
  uniformly for $\tau \in [-\tau_1,0)$ as $t \to T$,
  we see that $u(0,t+\lambda(t)^2 \tau) \ge 1$ 
  for all $\tau \in [-\tau_1,0)$ and $t$ close to $T$.
  This shows that Lemma~\ref{lemma:5.6} can be applied for 
  $(\eta_1,\eta_2)=(u(0,t+\lambda(t)^2 \tau),u(0,t))$ 
  or $(\eta_1,\eta_2)=(u(0,t),u(0,t+\lambda(t)^2 \tau))$.
  We hence have
  \begin{equation*}
    \begin{aligned}
      \frac{u(0,t+\lambda(t)^2 \tau)}{u(0,t)} 
      &\ge
      \min\left\{
        \left( 
          \frac{F(u(0,t))}{F(u(0,t+\lambda(t)^2 \tau))}
        \right)^{1/\alpha_1}, 
        \left( 
          \frac{F(u(0,t))}{F(u(0,t+\lambda(t)^2 \tau))}
        \right)^{1/\alpha_2}
      \right\}
      \\
      &=
      \min\left\{
        \left( 
          \frac{\lambda(t)^2}{\lambda(t+\lambda(t)^2 \tau)^2}
        \right)^{1/\alpha_1},
        \left( 
          \frac{\lambda(t)^2}{\lambda(t+\lambda(t)^2 \tau)^2}
        \right)^{1/\alpha_2}
      \right\}
      \\
      &\ge 
      \min\left\{
        \frac{1}{(1+|\tau|)^{1/\alpha_1}},
        \frac{1}{(1+|\tau|)^{1/\alpha_2}}
      \right\}
      =\frac{1}{(1+|\tau|)^{1/\alpha_1}},
    \end{aligned}
  \end{equation*}
  where we have used \eqref{eq:5.9} in the second inequality.
  Combining this with \eqref{eq:5.34},
  we deduce that
  \begin{align*}
    \liminf_{t \to T} \inf_{|y| \le \rho_*, |\tau|\le \tau_1} w^{t}(y,\tau) 
    &=\liminf_{t \to T} \inf_{|y| \le \rho_*, |\tau|\le \tau_1}
    \frac{u(0,t+\lambda(t)^2 \tau)}{u(0,t)} 
    \cdot \frac{u(\lambda (t)y,t+\lambda(t)^2 \tau)}{u(0,t+\lambda(t)^2 \tau)}
    \\
    &\ge \frac{1}{2(1+\tau_1)^{1/\alpha_1}}.
  \end{align*}
  Therefore \eqref{eq:5.33} is verified.
    
  Now we take a function $\zeta \in C^\infty([-\tau_1,\infty))$ satisfying
  \begin{equation*}
    \zeta(-\tau_1)=0,
    \qquad
    0<\zeta(\tau) \le \frac{1}{2}\underline{w}_0
    \quad \mbox{for} \quad \tau>-\tau_1,
  \end{equation*}
  and consider a solution $\underline{w}$ of the problem
  \begin{equation*}
    \left\{
    \begin{aligned}
      &\partial_\tau \underline{w} =\Delta_y \underline{w},
      &&(y,\tau) \in (B_{\rho_0+1} \setminus \overline{B_{\rho_*}}) 
      \times (-\tau_1,\infty), 
      \\
      &\underline{w}=0,
      &&(y,\tau) \in \partial B_{\rho_0+1} \times (-\tau_1,\infty), 
      \\
      &\underline{w}=\zeta(\tau),
      &&(y,\tau) \in \partial B_{\rho_*} \times (-\tau_1,\infty), 
      \\
      &\underline{w}|_{\tau=-\tau_1}=0,
      &&y \in \overline{B_{\rho_0+1}} \setminus B_{\rho_*}.
    \end{aligned}
    \right.
  \end{equation*}
  By \eqref{eq:5.33}, 
  we see that if $t$ is close to $T$,
  \begin{equation*}
    w^t(y,\tau) \ge \frac{1}{2}\underline{w}_0 \ge \zeta(\tau) =\underline{w}(y,\tau)
    \quad \mbox{for} \quad 
    (y,\tau) \in \partial B_{\rho_*} \times [-\tau_1,\tau_1].
  \end{equation*}
  This with the positivity of $w^t$ yields $w^{t} \ge \underline{w}$ 
  on the parabolic boundary of 
  $(B_{\rho_0+1} \setminus \overline{B_{\rho_*}}) \times [-\tau_1,\tau_1]$. 
  Hence we see from the maximum principle that
  \begin{equation*}
    w^{t}(y,\tau) \ge \underline{w} (y,\tau)
    \quad \mbox{for} \quad 
    (y,\tau) \in (\overline{B_{\rho_0+1}} \setminus B_{\rho_*}) \times [-\tau_1,\tau_1].
  \end{equation*}
  Since the strong maximum principle shows that $\underline{w}$ is positive 
  on $(B_{\rho_0+1} \setminus B_{\rho_*}) \times (-\tau_1,\infty)$,
  we infer that
  \begin{equation*}
    \inf_{\rho_* \le |y| \le \rho_0, |\tau|\le \tau_0} w^{t}(y,\tau) 
    \ge \inf_{\rho_* \le |y| \le \rho_0, |\tau|\le \tau_0} \underline{w}(y,\tau)
    >0.
  \end{equation*}
  From this and \eqref{eq:5.33}, 
  we conclude that
  \begin{equation}
    \label{eq:5.35}
    \liminf_{t \to T} \inf_{|y| \le \rho_0, |\tau|\le \tau_0} w^{t}(y,\tau)>0.
  \end{equation}

  We see from \eqref{eq:5.35} that
  \begin{equation*}
    u(\lambda(t)y,t+\lambda(t)^2\tau)
    =w^t(y,\tau) \cdot u(0,t) \to \infty
  \end{equation*}
  uniformly for $(y,\tau) \in \overline{B_{\rho_0}} \times [-\tau_0,\tau_0]$ as $t \to T$.
  In particular, if $t$ is close to $T$, 
  then $u(\lambda(t)y,t+\lambda(t)^2\tau) \ge 1$
  for all $(y,\tau) \in \overline{B_{\rho_0}} \times [-\tau_0,\tau_0]$.
  Hence we can apply Lemma~\ref{lemma:5.6} 
  for $(\eta_1,\eta_2)=(u(\lambda(t)y,t+\lambda(t)^2\tau),u(0,t))$ 
  or for $(\eta_1,\eta_2)=(u(0,t),u(\lambda(t)y,t+\lambda(t)^2\tau))$
  to obtain
  \begin{equation*}
    v^t(y,\tau) 
    \le 
    \max\left\{ w^t(y,\tau)^{-\alpha_1}, w^t(y,\tau)^{-\alpha_2} \right\}
    \quad \mbox{for} \quad
    (y,\tau) \in \overline{B_{\rho_0}} \times [-\tau_0,\tau_0].
  \end{equation*}
  The upper estimate of $v^t$ in \eqref{eq:5.8} thus follows from \eqref{eq:5.35}.
  
  What is left is to derive the estimate for $|\nabla v^{t}|$.
  We observe that \eqref{eq:1.7} gives 
  $\sup\limits_{\eta \ge 1} f_0'(\eta)F_0(\eta) \le \gamma$ for some $\gamma>1$.
  This means that \eqref{eq:5.24} holds 
  if $\kappa(t)$ and $m(t)$ are replaced by $\gamma$ and $1$, respectively.
  Hence the inequality \eqref{eq:5.25}, 
  with $\kappa(t)$ replaced by $\gamma$,
  is valid under the condition
  \begin{equation}
    \label{eq:5.36}
    u(z,s) \ge 1.
  \end{equation}
  From \eqref{eq:5.6}, \eqref{eq:5.8} and the definition of $v^t$, 
  we see that
  \begin{equation*}
    F(u(\lambda(t)y,t+\lambda(t)^2\tau))
    =v^t(y,\tau) \cdot F(u(0,t)) \to 0
  \end{equation*}
  uniformly for $(y,\tau) \in \overline{B_{\rho_0}} \times [-\tau_0,\tau_0]$
  as $t \to T$.
  This shows that the condition \eqref{eq:5.36},
  which is equivalent to $F(u(z,s)) \le F(1)$,
  is satisfied if $(z,s)=(\lambda (t)y,t+\lambda (t)^2 \tau)$
  $|y| \le \rho_0$, $|\tau| \le \tau_0$ and $t$ is close to $T$.
  Hence, by the same argument as in the derivation of \eqref{eq:5.31},
  we obtain
  \begin{equation*}
    \limsup_{t \to T} \sup_{|y| \le \rho_0, |\tau| \le \tau_0} 
    |\nabla v^{t}(y,\tau)| 
    \le 
    \frac{4^{1+\gamma}}{(\gamma-1/2)^{1/2}} \cdot 
    \left( 
    \frac{1}{1-\tau_0} 
    \right)^{\gamma-1/2} 
    \cdot K_0^{\gamma}.
  \end{equation*}  
  We have thus shown \eqref{eq:5.8},
  and the proof is complete.
\end{proof}

\subsection{Proof of Proposition~\ref{proposition:2.1}}

We conclude the paper by proving Proposition~\ref{proposition:2.1}.

\begin{proof}[Proof of Proposition~\ref{proposition:2.1}]
  We prove the proposition by contradiction.
  Contrary to our claim,
  suppose that
  \begin{equation}
    \label{eq:5.37}
    \liminf_{t\to T} \sup_{s \in (t,T)} 
    \frac{F(M(t))-F(M(s))}{s-t}=0
    \quad\mbox{and}\quad 
    \limsup_{t\to T} M(t)= \infty.
  \end{equation}
  One can then take a monotonically increasing sequence $\{t_i\}_{i=1}^\infty \subset [t_1 ,T)$ 
  such that $t_i \to T$ and
  \begin{equation}
    \label{eq:5.38}
    \delta_i:=\sup_{s \in (t_i,T)} 
    \frac{F(M(t_i))-F(M(s))}{s-t_i}
    \to 0
    \quad \mbox{as} \quad
    i\to\infty.
  \end{equation}
  Moreover, since \eqref{eq:1.8} and Lemma~\ref{lemma:5.1} yield 
  \begin{equation}
    \label{eq:5.39}
    M'(t)=\partial_t u(0,t) \neq 0
    \quad \mbox{for} \quad
    t \in [T_*,T),
  \end{equation}
  by \eqref{eq:5.37} we have
  \begin{equation}
    \label{eq:5.40}
    \lim_{t \to T} u(0,t) 
    =\lim_{t \to T} M(t)= \infty.
  \end{equation}

  Let us introduce a rescaled function. 
  Set
  \begin{equation*}
    g_q(\eta):=
    \left\{
    \begin{aligned} 
      &e^\eta
      &&\mbox{if}\quad q=1,
      \\
      &\eta^{\frac{q}{q-1}}
      &&\mbox{if}\quad q>1,
    \end{aligned}
    \right.
    \qquad
    G_q(\eta):=\int_\eta^\infty \frac{1}{g_q(\zeta)}d\zeta=
    \left\{
    \begin{aligned} 
      &e^{-\eta}
      &&\mbox{if}\quad q=1,
      \\
      &(q-1)\eta^{-\frac{1}{q-1}}
      &&\mbox{if}\quad q>1,
    \end{aligned}
    \right.
  \end{equation*}
  and take $f_0$ satisfying \eqref{eq:1.7} and \eqref{eq:5.14}.
  We then define
  \begin{equation*}
    w_i(y,\tau) :=G_q^{-1}
    \left(
    \frac{F_0(u(\lambda_i y, t_i+\lambda_i^2\tau))}{\lambda_i^2}
    \right),
    \qquad
    \lambda_i :=F(u(0,t_i))^\frac{1}{2}.
  \end{equation*}  
  To derive an equation for $w_i$,
  we differentiate the relation 
  \begin{equation}
    \label{eq:5.41}
    G_q(w_i(y,\tau))=\frac{F_0(u(\lambda_i y, t_i+\lambda_i^2\tau))}{\lambda_i^2}
  \end{equation}
  with respect to $y$ and $\tau$.
  Then we have
  \begin{equation}
    \label{eq:5.42}
    G_q'(w_i)\partial_\tau w_i
    =-\frac{\partial_t u}{f_0(u)}, 
    \qquad 
    G_q'(w_i)\nabla_y w_i 
    =-\frac{\nabla_x u}{\lambda_i f_0(u)}, 
  \end{equation}
  where $(\lambda_i y,t_i+\lambda_i^2\tau )$ is substituted for $(x,t)$ in $u=u(x,t)$.
  We again differentiate the latter of \eqref{eq:5.42} with respect to $y$ to obtain
  \begin{align*}
    G_q'(w_i)\Delta_y w_i +G_q''(w_i) |\nabla_y w_i|^2
    &=-\frac{\Delta_x u}{f_0(u)}
    +\frac{f_0'(u)}{f_0(u)^2} |\nabla_x u|^2
    \\
    &=-\frac{\Delta_x u}{f_0(u)}
    +\frac{f_0'(u)F_0(u)G_q'(w_i)^2}{G_q(w_i)} |\nabla_y w_i|^2,
  \end{align*}
  where we have used \eqref{eq:5.41} and \eqref{eq:5.42}
  in deriving the last inequality.
  Since a direct computation gives
  \begin{equation*}
    \frac{G_q'(\eta)}{G_q(\eta)}
    =
    \left\{
    \begin{aligned}
    &-1
    &&\mbox{if}\quad q=1,
    \\
    &-\frac{1}{(q-1)\eta}
    &&\mbox{if}\quad q>1,
    \end{aligned}
    \right. 
    \qquad
    \frac{G_q''(\eta)}{G_q'(\eta)}
    =
    \left\{
    \begin{aligned}
    &-1
    &&\mbox{if}\quad q=1,
    \\
    &-\frac{q}{(q-1)\eta}
    &&\mbox{if}\quad q>1,
    \end{aligned}
    \right. 
  \end{equation*}
  we conclude that $w_i$ satisfies
  \begin{equation}
    \label{eq:5.43}
    \begin{aligned}
      \partial_\tau w_i -\Delta_y w_i
      &=-\frac{\partial_t u -\Delta_x u}{f_0(u)G_q'(w_i)}
      +\left(
      -\frac{f_0'(u)F_0(u)G_q'(w_i)}{G_q(w_i)} +\frac{G_q''(w_i)}{G_q'(w_i)}
      \right)
      |\nabla_y w_i|^2
      \\
      &=\left\{
      \begin{aligned}
        &\frac{f(u)}{f_0(u)} g_q(w_i) + (f_0'(u)F_0(u)-1)|\nabla_y w_i|^2
        &&\mbox{if}\quad q=1,
        \\
        &\frac{f(u)}{f_0(u)} g_q(w_i) + \frac{(f_0'(u)F_0(u)-q) |\nabla_y w_i|^2}{(q-1)w_i}
        &&\mbox{if}\quad q>1.
      \end{aligned}
      \right.
    \end{aligned} 
  \end{equation}

  To derive a contradiction, 
  we first consider the convergence of the sequence $\{w_i\}$.
  From \eqref{eq:5.14},
  we have
  \begin{equation*}
    G_q^{-1}
    \left(
    2v^{t_i}(y,\tau)
    \right)
    \le
    w_i(y,\tau) = G_q^{-1}
    \left(
    \frac{F_0(u(\lambda(t_i) y, t_i+\lambda(t_i)^2\tau))}{\lambda(t_i)^2}
    \right)
    \le 
    G_q^{-1}
    \left(
    \frac{v^{t_i}(y,\tau)}{2}
    \right),
  \end{equation*}
  where $v^t(y,\tau)$ and $\lambda(t)$ are the functions defined in the previous subsection.
  This together with Proposition~\ref{proposition:5.1} shows that
  \begin{equation}
    \label{eq:5.44}
    \limsup_{i\to\infty}
    \sup_{(y,\tau)\in \overline{B_{\rho_0}} \times [-\tau_0,\tau_0]}
    (|w_i(y,\tau)|+|w_i(y,\tau)^{-1}|+|\nabla w_i(y,\tau)|) < \infty 
  \end{equation}
  for any $\rho_0>0$ and $\tau_0 \in (0,1)$.
  By a regularity theory for parabolic equations (see for example \cite{LSU}), 
  we see that $\{w_i\}$ is relatively compact 
  in $C(\overline{B_{\rho_0}} \times [-\tau_0,\tau_0])$.
  Therefore there exist a subsequence of $\{w_i\}$,
  which is still denoted by $\{w_i\}$, 
  and a radially symmetric function $w \in C(\mathbb{R}^N \times (-1,1))$ 
  such that
  \begin{equation}
    \label{eq:5.45}
    w_i \to w
    \quad \mbox{in} \quad C_{\operatorname{loc}}(\mathbb{R}^N \times (-1,1))
  \end{equation}
  as $i \to \infty$.
  
  Next we examine properties of $w$.
  Notice that
  \begin{equation}
    \label{eq:5.46}
    u(\lambda_i y,t_i+\lambda_i^2 \tau) \to \infty
    \quad \mbox{locally uniformly for} \,\,\, 
    (y,\tau) \in \mathbb{R}^N \times (-1,1)
    \,\,\, \mbox{as} \,\,\, 
    i \to \infty.
  \end{equation}
  Indeed, this is shown by writing
  \begin{equation*}
    u(\lambda_i y,t_i+\lambda_i^2 \tau)
    =F^{-1} \left(
    v^{t_i}(y,\tau) F(u(0,t_i))
    \right)
  \end{equation*}
  and then applying \eqref{eq:5.40} and Proposition~\ref{proposition:5.1}.
  Hence, letting $i \to \infty$ in \eqref{eq:5.43} 
  and using \eqref{eq:1.7}, \eqref{eq:5.44}, \eqref{eq:5.45} and \eqref{eq:5.46},
  we find that $w$ satisfies
  \begin{equation}
    \label{eq:5.47}
    \qquad
    \partial_\tau w -\Delta_y w=g_q(w)
  \end{equation}
  in the distributional sense.
  Since $g_q$ is smooth away from zero,
  we see from a regularity theory for parabolic equations 
  that $w$ is smooth and satisfies \eqref{eq:5.47} in the classical sense.
  We now show that $w$ also satisfies
  \begin{equation}
    \label{eq:5.48}
    \partial_\tau w(0,\tau)=0
    \quad \mbox{for} \quad
    \tau \in (0,1).
  \end{equation}
  To this end,
  we observe that
  \begin{equation}
    \label{eq:5.49}
    \frac{1}{2} \le 
    \frac{F(\eta_1)-F(\eta_2)}{F_0(\eta_1)-F_0(\eta_2)}
    \le 2
    \quad \mbox{if} \quad
    \eta_1 \neq \eta_2.
  \end{equation}
  This is true because Cauchy's mean value theorem yields
  \begin{equation*}
    \frac{F(\eta_1)-F(\eta_2)}{F_0(\eta_1)-F_0(\eta_2)}
    =\frac{f_0(\eta_2+\theta (\eta_1-\eta_2))}{f(\eta_2+\theta (\eta_1-\eta_2))}
    \quad \mbox{for some} \quad \theta \in (0,1),
  \end{equation*}
  and \eqref{eq:5.14} shows that the right-hand side of this equality lies between $1/2$ and $2$.
  We define
  \begin{equation*}
    \tilde \delta_{i,\tau} :=
    \frac{F_0(M(t_i))-F_0(M(t_i+\lambda_i^2\tau))}{\lambda_i^2 \tau}
  \end{equation*}
  for $\tau \in (0,1)$.
  Then $\tilde \delta_{i,\tau}>0$ for large $i$, 
  since \eqref{eq:5.39} and \eqref{eq:5.40} give $M'(t)>0$ if $t$ is close to $T$.
  Moreover, from the definition of $\delta_i$ and \eqref{eq:5.49},
  we have $\tilde \delta_{i,\tau} \le 2\delta_i$.
  Hence by \eqref{eq:5.38} we infer that
  \begin{equation*}
    \lim_{i \to \infty} \tilde \delta_{i,\tau} =0.
  \end{equation*}
  We note that the definition of $w_i$ gives
  \begin{equation*}
    G_q(w_i(0,0))- G_q(w_i(0,\tau))
    =\frac{F_0(u(0,t_i))-F_0(u(0,t_i+\lambda_i^2\tau))}{\lambda_i^2}
    =\tau \tilde \delta_{i,\tau}.
  \end{equation*}
  Therefore by letting $i \to \infty$ we obtain
  \begin{equation*}
    G_q(w(0,0))- G_q(w(0,\tau))=0,
  \end{equation*}
  which results in \eqref{eq:5.48}.

  Finally we derive a contradiction.
  From \eqref{eq:5.47} and \eqref{eq:5.48},
  Lemma~\ref{lemma:5.2} can be applied.
  We hence have $\partial_\tau w=0$ in $\mathbb{R}^N \times (0,1)$.
  In particular, by \eqref{eq:5.47},
  we see that $\Phi:=W(\cdot,0)$ is a solution of $\Phi''+(N-1)\Phi'/r +g_q(W)=0$.
  In other words,
  $\Phi$ satisfies \eqref{eq:4.1} with $p=q/(q-1)$ if $q>1$ and \eqref{eq:4.2} if $q=1$.
  Since using \eqref{eq:1.7}, \eqref{eq:5.40}, \eqref{eq:5.45} and L'Hospital's rule give
  \begin{equation*}
    G_q(\Phi(0))=G_q(w(0,0))
    =\lim_{i \to \infty} G_q(w_i(0,0))
    =\lim_{i \to \infty} \frac{F_0(u(0, t_i))}{F(u(0, t_i))}
    =\lim_{i \to \infty} \frac{f(u(0, t_i))}{f_0(u(0, t_i))}
    =1,
  \end{equation*}
  $\Phi$ also satisfies
  \begin{equation}
    \notag 
    \Phi(0)
    =\left\{
    \begin{aligned} 
      &0
      &&\mbox{if}\quad q=1,
      \\
      &(q-1)^{q-1}
      &&\mbox{if}\quad q>1.
    \end{aligned}
    \right.
  \end{equation}
  Therefore, applying Propositions~\ref{proposition:4.3} and \ref{proposition:4.4},
  we have
  \begin{equation}
    \label{eq:5.50}
    Z_{(0,\infty)}[\Phi-\Phi_p^*]=\infty,
    \qquad
    p=\left\{
    \begin{aligned}
      &\frac{q}{q-1}
      &&\mbox{if} \quad q>1,
      \\
      &\infty
      &&\mbox{if} \quad q=1.
    \end{aligned}
    \right.
  \end{equation}
  We now recall that 
  \begin{equation*}
    \Phi_p^* (\rho)
    =G_q^{-1} 
    \left(
    \frac{\rho^2}{2N-4q}
    \right).
  \end{equation*}
  Hence by Proposition~\ref{proposition:3.1} we have
  \begin{equation*}
    G_q^{-1} 
    \left(
    \frac{F_0(U^*(\lambda_i \rho))}{\lambda_i^2}
    \right)
    =G_q^{-1} 
   \left(
    \frac{\rho^2 (1+\theta(\lambda_i \rho))}{2N-4q}
    \right)
    \to G_q^{-1} 
    \left(
    \frac{\rho^2}{2N-4q}
    \right)
    =\Phi_p^* (\rho) 
  \end{equation*}
  locally uniformly for $\rho>0$ as $i \to \infty$.
  By \eqref{eq:5.45},
  we also have
  \begin{equation*}
    G_q^{-1} 
    \left(
    \frac{F_0(U(\lambda_i \rho,t_i))}{\lambda_i^2}
    \right)
    =W_i(\rho,0)
    \to W(\rho,0)
    =\Phi(\rho) 
  \end{equation*}
  locally uniformly for $\rho>0$ as $i \to \infty$.
  Therefore for any $\rho_0>0$,
  \begin{align*}
    Z_{(0,\rho_0)}[\Phi-\Phi_p^*]
    &
    \le \liminf_{i \to \infty} Z_{(0,\rho_0)}
    \left[
    G_q^{-1} 
    \left( 
    \frac{F_0(U(\lambda_i \, \cdot,t_i))}{\lambda_i^2}
    \right)
    -G_q^{-1} 
    \left(
    \frac{F_0(U^*(\lambda_i \, \cdot))}{\lambda_i^2}
    \right)
    \right]
    \\
    &=
    \liminf_{i \to \infty} Z_{(0,\rho_0)}
    \left[
    U(\lambda_i \, \cdot,t_i)) -U^*(\lambda_i \, \cdot)
    \right],
  \end{align*}
  where we have used the fact that 
  \begin{gather*}
    Z_I[h] \le \liminf_{n \to \infty} Z_I[h_n]
  \end{gather*}
  for any $C^1$ function $h$ having only simple zeros on an interval $I$
  and any sequence $\{h_n\} \subset C(I)$ converging pointwise to $h$.
  This leads to a contradiction,
  because \eqref{eq:5.50} shows that 
  the left-hand side of this equality diverges to $\infty$ as $\rho_0 \to \infty$,
  while Lemma~\ref{lemma:5.1} shows that 
  the right-hand side does not exceed the constant $N_0$.
  Thus \eqref{eq:5.37} is false, and the proof is complete. 
\end{proof}

\bigskip 

\noindent
{\bf Acknowledgements.}
The first author was supported in part by JSPS KAKENHI Grant Number 23K03179.
The second author was partially supported by JSPS KAKENHI Grant Number 22H01131.



\end{document}